\documentclass[twoside,12pt]{article}
\usepackage[usenames,dvipsnames]{color}
\usepackage{amsmath,amsthm,amssymb}
\usepackage[mathcal]{euscript}
\usepackage{xcolor}
\usepackage{cite}
\usepackage{enumitem}
\usepackage{hyperref,url}

\def\last#1{{\color{blue}#1}}
\def\last#1{#1}

\theoremstyle{plain}
\newtheorem{thm}{Theorem}[section]
\newtheorem{rem}{Remark}[section]
\newtheorem{cor}[thm]{Corollary}
\newtheorem{defn}{Definition}[section]
\newtheorem{lem}[thm]{Lemma}
\numberwithin{equation}{section}

\def\Q{\mathcal{Q}}

\def\intQ{\int_Q}
\def\iO{\int_\Omega}

\def\dt{\partial_t}
\def\dn{\partial_\bold n}

\def\pdt{\partial_\tau}
\def\ph{\varphi}
\def\s{\sigma}
\def\div{\mathrm{div}}
\def\lam{\lambda}
\def\Lam{\Lambda}
\def\<#1>{\mathopen\langle #1\mathclose\rangle}
\def\mean #1{\overline{#1}}
\def\eps{\varepsilon}

\def\laminfty{\lambda_\infty}
\renewcommand\bar {\overline}

\begin{document}
\title{Strong well-posedness and
inverse identification problem
of a non-local phase field tumor model with degenerate
mobilities}
\author{}
\date{}
\maketitle

\begin{center}
\vskip-1cm
{\large\sc Sergio Frigeri$^{(1)}$}\\
{\normalsize e-mail: {\tt sergio.frigeri@unimi.it}}\\[.25cm]
{\large\sc Kei Fong Lam$^{(2)}$}\\
{\normalsize e-mail: {\tt akflam@hkbu.edu.hk}}\\[.25cm]
{\large\sc Andrea Signori$^{(3)}$}\\
{\normalsize e-mail: {\tt andrea.signori02@universitadipavia.it}}\\[.25cm]

$^{(1)}$
{\small Dipartimento di Matematica ``Federigo Enriques", Universit\`a degli Studi di Milano}\\
{\small via Saldini 50, I-20133 Milano, Italy}\\[.2cm]
$^{(2)}$
{\small Department of Mathematics, Hong Kong Baptist University}\\
{\small Kowloon Tong, Hong Kong}\\[.2cm]
$^{(3)}$
{\small Dipartimento di Matematica e Applicazioni, Universit\`a di Milano--Bicocca}\\
{\small via Cozzi 55, 20125 Milano, Italy}\\[.2cm]
\end{center}

\begin{abstract}
We extend previous weak well-posedness results obtained in Frigeri et al.~(2017) concerning a non-local variant of a diffuse interface tumor model proposed by Hawkins-Daarud et al.~(2012).  The model consists of a non-local Cahn--Hilliard equation with degenerate mobility and singular potential for the phase field variable, coupled to a reaction-diffusion equation for the concentration of a nutrient.  We prove the existence of strong solutions to the model and establish some high order continuous dependence estimates, even in the presence of concentration-dependent mobilities for the nutrient variable in two spatial dimensions.  Then, we apply the new regularity results to study an inverse problem identifying the initial tumor distribution from measurements at the terminal time.  Formulating the Tikhonov regularised inverse problem as a constrained minimisation problem, we establish the existence of minimisers and derive first-order necessary optimality conditions.
\vskip3mm
\noindent {\bf Key words:}
Tumor growth, non-local Cahn--Hilliard equation, degenerate mobility, singular potentials, strong solutions, well-posedness, inverse problem, G\^ateaux differentiability.
\vskip3mm
\noindent {\bf AMS Subject Classification:}
35K86,  
35K61,  
35Q92,  
49J20,  
92C50.  
\end{abstract}

\pagestyle{myheadings}
\newcommand\testopari{\sc Frigeri -- Lam -- Signori}
\newcommand\testodispari{\sc Regularity results for a tumor model with degenerate mobilities}
\markboth{\testopari}{\testodispari}

\newpage

\section{Introduction}
Mathematical modeling for tumor growth dynamics has undergone a swift development in the last decades (see for instance pioneering works such as \cite{CLLW,CL,WLFC}). Even now, the full complexity of the tumor disease is far from being understood, and through mathematical modeling, scientists and medical practitioners now possess a powerful tool to predict and analyse tumor growth behaviour without inflicting serious harm to the patients.

In this contribution, we address the issue of well-posedness for a certain continuum model for tumor growth.  The original model, derived in Hawkins-Daruud et al.~\cite{HDZO} (see also \cite{HDPZO,HKNZ}), is based on the well-known phase field methodology that has seen increased applications in tumor growth, and takes the form
\begin{align*}
\dt \ph  &= \div (m(\ph) \nabla \mu) + P(\ph)(\sigma + \chi(1-\ph) - \mu), \\
\mu & = A F'(\ph) - B \Delta \ph - \chi \s, \\
\dt \s & = \div (n(\ph) \nabla (\s + \chi(1-\ph))) - P(\ph)(\s + \chi(1-\ph) - \mu),
\end{align*}
where the primary variables $(\ph, \mu, \s)$ denote the phase field, the associated chemical potential, and the nutrient concentration, respectively.  The phase field $\ph$ serves as an indicator on the location of tumor and non-tumor cells, which are separated by a thin interfacial layer whose thickness is related to the positive constants $A$ and $B$, while the non-negative functions $m(\ph)$ and $n(\ph)$ correspond to the cellular and nutrient mobilities, respectively.  The function $F'(\ph)$ is the derivative of a potential $F(\ph)$, which is a characteristic feature of phase field models.  Lastly, the non-negative constant $\chi$ is a chemotactic sensitivity of the nutrient and $P(\ph)$ denotes a proliferation function, see  \cite{GLSS,HDZO} for more details.

The mathematical and numerical analysis, and optimal control for the above model of Hawkins-Daarud et al.~and its variants have been performed by many authors, of which we mention \cite{CRW,CGH,CGRS_VAN,CGRS_OPT,CGRS_ASY,CSS1,FGR,HKNZ,S_a,S_b,S_DQ,S,S_time,WZZ}.  Such intensive study and broad range of results are possible thanks to the Lyapunov structure of the model, where in a bounded domain $\Omega \subset \mathbb{R}^d$, under no-flux boundary conditions
$\dn \ph = m(\ph)\dn \mu = n(\ph)\dn \s = 0$ ($\dn f = \nabla f \cdot \bold n$ is the normal derivative) sufficiently smooth solutions satisfy
\begin{align*}
\frac{d}{dt}E(\ph(t),\s(t)) & + \int_0^t \|m^{1/2}(\ph) \nabla \mu\|^2 + \| n^{1/2}(\ph) \nabla (\s - \chi \ph)\|^2 \\
& + \int_0^t \| P^{1/2}(\ph)(\s + \chi(1-\ph) - \mu) \|^2  = 0,
\end{align*}
with the free energy function
\begin{align*}
E(\ph, \s) :=\mathcal{L}(\ph) + \int_\Omega  \frac{1}{2} |\s|^2 + \chi \s(1-\ph), \quad \mathcal{L}(\ph) := \int_\Omega A F(\ph) + \frac{B}{2} |\nabla \ph|^2.
\end{align*}
\last{The above energy equality can be} \last{formally derived by summing} \last{the resulting equalities obtained}
\last{by multiplying} \last{the three equations in the tumor model with $\mu$, $\dt \ph$ and $\s \last{+} \chi (1-\ph)$, respectively,}
\last{and by integrating them over $\Omega$}.
In the above, $\mathcal{L}(\ph)$ is the Ginzburg--Landau energy function, which is responsible for phase separation and surface tension effects in the context of phase field models.  In our current context of tumor growth, $\mathcal{L}(\ph)$ is associated with cell-to-cell adhesion, where tumor cells prefer to adhere to each other rather than to non-tumor cells.

More recent studies have proposed to include fluid motion \cite{DFRGM,EGAR,EGAR2,EK_ADV,EK,FLRS,FLOW,GLNS,GLSS,JWZ,SW,WLFC,EGN,EL} and elasticity effects \cite{GLS_elastic,Lima1,Lima2} to better emulate in vivo tumor growth, where the cellular environment such as the presence of the extracellular matrix or rigid bone can assert significant influences on tumor proliferation.  In this work, we focus on a different aspect, where we replace the term $\frac{B}{2} |\nabla \ph(x)|^2$ in the Ginzburg--Landau energy function $\mathcal{L}(\ph)$ with a convolution $- \frac{B}{2} \int_\Omega J(x-y)\ph(x) \ph(y) dy$, leading to a non-local variant of the free energy function
\begin{equation}\label{nonloc:GL}
\begin{aligned}
\mathcal{E}(\ph, \s) &:= \mathcal{F}(\ph) + \int_\Omega \frac{1}{2} |\s|^2 + \chi \s(1-\ph), \\
 \mathcal{F}(\ph) &:= - \frac{B}{2} \int_\Omega \int_\Omega J(x-y) \ph(x) \ph(y) + A \int_\Omega F(\ph).
\end{aligned}
\end{equation}
\last{For a discussion on the motivation and physical relevance of the choice of this form for the non-local free energy function
$ \mathcal{F}$ we refer the reader to \cite[Introduction]{F2} (see also \cite{GGG}).}
Since its introduction by Giacomin and Lebowitz \cite{GL,GL1,GL2}, phase field models derived from the non-local Ginzburg--Landau energy $\mathcal{F}(\ph)$ have been the subject of intensive studies, \last{see, e.g., \cite{GZ,CFG,F1,F2,FGG2,FGGS,FG1,FG2,FGK,FGR2,FGS,GGG}.}
In our present context, the non-local energy function $\mathcal{F}(\ph)$ accounts for non-local cell-to-cell adhesion, compare also with \cite{APS,Chap,FLNOW,GC,SS}.  The resulting tumor model now reads as
\begin{subequations}\label{nonloc:Haw}
\begin{alignat}{2}
\dt \ph & = \div (m(\ph) \nabla \mu) + P(\ph)(\s + \chi(1-\ph) - \mu), \label{nonloc:Haw:1} \\
\mu & = A F'(\ph) - B J \star \ph - \chi \s, \label{nonloc:Haw:2} \\
\dt \s & = \div (n(\ph) \nabla(\s + \chi (1-\ph))) - P(\ph)(\s + \chi(1-\ph) - \mu), \label{nonloc:Haw:3}
\end{alignat}
\end{subequations}
where $J$ is a suitable spatial convolution kernel and
\begin{align*}
(J \star \ph)(x,t) := \int_\Omega J(x-y) \ph(y,t) dy\qquad\last{\forall (x,t) \in Q} := \Omega \times (0,T).
\end{align*}
In \cite{FLR}, the existence of weak solutions to \eqref{nonloc:Haw} for a wide range of non-degenerate mobility functions $m, n$, proliferating function $P$, and potential $F$ has been established by the first and second authors of this work.  Continuous dependence on initial data (and hence uniqueness of weak solutions) can be achieved under the additional requirement that $\chi = 0$ and $n = 1$.  Furthermore, by adapting the method introduced in \cite{EG} the authors in \cite{FLR} were able to establish weak well-posedness of the non-local tumor model \eqref{nonloc:Haw} when the mobility $m(s)$ is degenerate at $s = \pm 1$, and the potential $F(s)$ is singular at $s = \pm 1$, i.e., $\lim_{s \to \pm 1} F'(s) = \pm \infty$.  The prototypical example is
\begin{align}\label{m:F:eg}
m(s) = D(s)(1-s^2), \quad F(s) = (1-s) \log(1-s) + (1+s) \log(1+s)
\end{align}
for $s \in [-1,1]$ and a non-degenerate function $D$.  The need to consider such degenerate/singular terms in the tumor model arises from the fact that certain physical quantities such as tumor mass densities are only defined if the phase field variable $\ph$ belongs to the physical interval $ [-1,1]$.  This cannot be guaranteed, even at the continuous level, if one employs a smooth potential $F$, such as the classical quartic function $F(s) = (s^2-1)^2$.

On the other hand, the presence of these degenerate/singular terms \last{limits} the analytical investigations of \eqref{nonloc:Haw} to the class of weak solutions, and to the best of the authors' knowledge, numerical analysis and optimal control involving the non-local model \eqref{nonloc:Haw} with degenerate mobility and singular potentials have not received much attention in the literature. Therefore, the purpose of this work is to prove the well-posedness of strong solutions in order to facilitate future investigations.

In the following, we consider a bounded domain $\Omega \subset \mathbb{R}^d$, $d \in \{2,3\}$ with Lipschitz boundary $\Gamma := \partial \Omega$.  For a fixed but arbitrary constant $T > 0$ we denote the parabolic cylinder and its boundary by $Q_t := \Omega \times (0,t)$ and $\Sigma_t := \Gamma \times (0,t)$ for all $t \in (0,T)$, with $Q := Q_T$ and $\Sigma := \Sigma_T$.  In light of previous results in \cite{FLR} we switch off the chemotaxis mechanism by setting $\chi = 0$, and owing to the degeneracy of the mobility $m$ the gradient of the chemical potential $\mu$ which appears in equation \eqref{nonloc:Haw:1} cannot be controlled in any Lebesgue space.  Thus, following \cite{B,EG,FGGS,FGR2,GZ,GL1,GL2}, we introduce the auxiliary function
\begin{align}\label{definition_lam}
\lam (s):= m(s)F''(s),	 \quad \Lam (s):=A\int_0^s \lam(r) dr \qquad\last{\forall  s\in[-1,1]},
\end{align}
which exhibits the following useful relations
\begin{align}
\label{prop_lam}
\nabla \Lam (\ph) =A \lam(\ph)\nabla \ph, \quad	\dt \Lam (\ph) =A \lam(\ph)\dt\ph,
\end{align}
and upon substituting \eqref{nonloc:Haw:2} into \eqref{nonloc:Haw:1} and \eqref{nonloc:Haw:3}, we arrive at the following strong formulation of non-local model \eqref{nonloc:Haw}:
\begin{subequations}\label{new:model}
\begin{alignat}{2}
 \notag \dt \ph - \Delta \Lam (\ph) &= - B \, \div ( m(\ph) (\nabla J\star \ph) )  \\
 & \quad + P(\ph)(\s - AF'(\ph)+ B J\star\ph ) \quad \text{ in } Q, \label{eq_first} \\
 \dt \s - \div (n(\ph)\nabla \s) &= - P(\ph)(\s - AF'(\ph)+ B J\star\ph ) \quad \text{ in } Q .\label{eq_second}
\end{alignat}
\end{subequations}
For boundary conditions, we take the no-flux conditions $m(\ph) \dn \mu = n(\ph) \dn \s = 0$, which translate to
\begin{align}\label{bc}
[\nabla \Lam(\ph) - {B}m(\ph)(\nabla J \star \ph)]\cdot \bold{n} = 0, \quad n(\ph)\dn\s=0 \quad \text{ on } \Sigma,
\end{align}
and for initial conditions, we prescribe
\begin{align}\label{ic}
\ph(x,0) = \ph_0(x), \quad \s(x,0) = \s_0(x)\quad \text{ for } x \in \Omega.
\end{align}
Since the weak well-posedness to \eqref{new:model}-\eqref{ic}, which we collectively call $(\mathrm{P})$, is a direct consequence of the main results of \cite{FLR}, the focus of this work is to show the existence of strong solutions using techniques inspired by \cite{FGGS} for the non-local Cahn--Hilliard--Navier--Stokes system.  In our setting, this involves a bootstrapping argument in which we first improve the regularity of $\ph$ by fixing $\s$ and employing a time discretisation of \eqref{eq_first}, and then we improve the regularity of $\s$ with the help of new regularities for $\ph$.  Under suitable assumptions detailed in the next section, our main results are $H^1(0,T;L^2(\Omega)) \cap L^\infty(0,T;H^1(\Omega)) \cap L^2(0,T;H^2(\Omega))$-regularities for $\ph$ and $\s$ (see Theorem \ref{thm:strong}) for $d \in \{2,3\}$ and $W^{1,\infty}(0,T;L^2(\Omega)) \cap H^1(0,T;H^1(\Omega)) \cap L^\infty(0,T;H^2(\Omega))$-regularities for $\ph$ and $\s$ with general assumptions for $d = 2$,
whereas for $d=3$ solely under the additional requirements that the nutrient mobility $n=1$ and $\lam$ is a positive constant (see Theorem \ref{thm:superstrong}).  In turn, these regularities lead to continuous dependence on initial data in stronger norms (see Theorems \ref{thm:cts1} and \ref{thm:cts2}) compared to those established in \cite{FLR}.
It is also worth mentioning the arguments of \cite{FGGS} do not apply directly to our model due to the presence of the proliferation term $P(\ph)(\s - A F'(\ph) + B J \star \ph)$ in \eqref{eq_first}, and some crucial parts of the argument have to be modified in order for the analysis to go through.

As an application of the new solution regularities, we study an inverse problem relevant to tumor growth, which involves identifying the initial tumor distribution encoded by $\ph_0$ based on measurements of the phase field variable at terminal time $\ph(T)$. Thanks to the well-posedness of $(\mathrm{P})$ we can introduce the notion of a solution operator $S : \ph_0 \mapsto \ph(T)$.  Then, given a measurement $\ph_\Omega : \Omega \to \mathbb{R}$ of the phase field variable, the inverse problem can be formulated as:
\begin{align}\label{inv}
\text{ Find } \ph_0 \text{ such that } S(\ph_0) = \ph_\Omega \text{ a.e.~in } \Omega.
\end{align}
Due to the compactness of the solution operator $S : H^1(\Omega) \to H^1(\Omega)$, the inverse problem is ill-posed \cite[Chap.~10]{EHN}.
To overcome this, we employ Tikhonov regularisation and formulate the resulting problem as a constrained minimisation problem.  More precisely, we employ optimal control methods treating $\ph_0$ as the optimal control to the problem
\begin{align}\label{opt}
\ph_0 = \mathrm{argmin}_{u \in U}\Big ( \frac{1}{2} \| S(u) - \ph_\Omega \|_{L^2(\Omega)}^2
+ \frac{\alpha}{2} \| u \|_{H^1(\Omega)}^2 \Big ),
\end{align}
where $U$ denotes a suitable set of admissible controls and $\alpha > 0$ is a regularisation parameter.  Our main results for \eqref{opt} are (i) the existence of a solution $\bar \ph_0^\alpha \in U$ for any $\alpha > 0$, (ii) how to obtain a solution to the inverse problem \eqref{inv} from $\{ \bar \ph_0^\alpha \}_{\alpha > 0}$ as $\alpha \to 0$ (provided the solution set of \eqref{inv} is non-empty), and (iii) the derivation of first-order optimality conditions for $\bar \ph_0^\alpha$.  The precise formulation can be found in Theorem \ref{thm:opt}.  In particular, thanks to the new solution regularities to $(\mathrm{P})$, practitioners interested in solving the inverse identification problem \eqref{inv} that involve the non-local tumor model \eqref{new:model} with degenerate mobility and singular potentials can first obtain numerical approximations of $\{\bar \ph_0^\alpha\}_{\alpha > 0}$ by solving the optimality conditions, and then sending $\alpha \to 0$ in an appropriate way to deduce a solution to \eqref{inv}.

The paper is structured as follows:  In Section~\ref{sec:setting} we state the notation and recall previous results on $(\mathrm{P})$, and in Section~\ref{sec:strong} we state and prove strong well-posedness to $(\mathrm{P})$.  In Section~\ref{sec:appl} we study the optimal control problem \eqref{opt} and derive desirable properties involving minimisers and the first-order optimality conditions.

\section{Mathematical setting and previous results}\label{sec:setting}
In this section, we recall some useful mathematical tools and previous results on $(\mathrm{P})$ established in \cite{FLR}.  We define
\begin{align}
H:= L^2(\Omega), \quad V:= H^1(\Omega), \quad W := H^2(\Omega),
\end{align}
and equip them with their standard norms.
Moreover, for an arbitrary Banach space $X$,
we indicate with $\|\cdot\|_X$, $X^*$, and $\< \cdot, \cdot >_X$
its norm, its topological dual and the duality pairing between $X^*$ and $X$, respectively.
Likewise, for every $1\leq p \leq \infty$, we simply use $\|\cdot\|_{p}$
to denote the usual norm in $L^p(\Omega)$, with $\|\cdot\| = \|\cdot\|_2$.  Furthermore, we use $(\cdot, \cdot)$ to denote the $L^2(\Omega)$-inner product.
As $(V,H,V^*)$ forms a Hilbert triplet, i.e., the injections $V \subset H \equiv H^* \subset V^*$
are both continuous and dense, we have the following identification
\begin{align*}
	\< u,v >_V = \iO uv
	\quad \last{\forall u\in H,\,\,\,\forall v\in V.}
\end{align*}
For $u \in L^1(\Omega)$, we use the notation $\mean{u} = \tfrac{1}{|\Omega|}(u,1)$ to denote the mean value of $u$.  \last{The Gagliardo--Nirenberg interpolation inequality in dimension $d$ (see \cite[Thm.~2.1]{DiBe} or \cite{BreMi}) is stated} \last{as follows.}
\last{Let $\Omega$ be a bounded domain with Lipschitz boundary and $f \in W^{m,r}(\Omega) \cap L^q(\Omega)$, $1 \leq q,r \leq \infty$.  For any integer $j$, $0 \leq j < m$, suppose there is an $\alpha \in \mathbb{R}$ such that
\begin{align*}
\frac{1}{p} = \frac{j}{d} + \left ( \frac{1}{r} - \last{\frac{m}{d}} \right ) \alpha + \frac{1-\alpha}{q} , \quad \frac{j}{m} \leq \alpha \leq 1.
\end{align*}
Then, there exists a positive constant $C$ depending only on $\Omega$, $m$, $j$, $q$, $r$, and $\alpha$ such that
\begin{align}
\| D^j f \|_{\last{L^p(\Omega)}} \leq C \| f \|_{W^{m,r}\last{(\Omega)}}^\alpha \| f \|_{L^q\last{(\Omega)}}^{1-\alpha}.\label{e19}
\end{align}
}
The following particular case of the Gagliardo--Nirenberg inequality in two dimensions will be repeatedly employed throughout our analysis
\begin{align}\label{interpol_abs_2d}
\|f\|_4 & \leq C \|f\|^{1/2} \|f \|_V^{1/2} \qquad\forall f \in V.
\end{align}
Lastly, we also recall the Agmon's inequality in two dimensions, see e.g.~\cite[Lem.~13.2]{Agmon}:
\begin{align}
\| f \|_{\infty} \leq C \|f \|^{1/2} \| f \|^{1/2}_{W}
	\qquad \forall f \in W.
	\label{agmon}
\end{align}
Throughout the paper, we will use the symbol $C$ to denote constants which depend only on structural data of the problem.  On the other hand, we will sometimes stress the dependence of the appearing constant by adding a self-explanatory subscript. Moreover, $\mathbb{Q}\geq 0$ will stand for a generic monotone non-decreasing continuous function of all its arguments.

For the analysis, we make the following structural assumptions:
\begin{enumerate}[label=$(\mathrm{A \arabic*})$, ref = $\mathrm{A \arabic*}$, series=A]
\item \label{ass:m:F} $m\in C^0([-1,1])$ and $F \in C^2(-1,1)$ with
\begin{align*}
m(s)> 0, \quad F''(s) \geq 0\quad \forall s \in (-1,1), \\
 m(\pm 1)=0,
\quad \lam :=mF'' \in C^0([-1,1]),
\end{align*}
\last{the last condition meaning that $mF''$ can be extended by continuity to the closed interval $[-1,1]$.} Moreover, there exist constants $\eps_0\in(0,1]$ and $\alpha_0 > 0$ such that $m$ is non-increasing in $[1-\eps_0,1]$ and non-decreasing in $[-1,-1+\eps_0]$,  $F''$ is non-decreasing in $[1-\eps_0,1)$ and non-increasing in $(-1,-1+\eps_0]$, and $\lambda(s) \geq \alpha_0$ for all $s \in [-1,1]$.
\item \label{ass:n} $n\in C^0([-1,1])$ and there exist
a positive constant $n_*$ such that
\begin{align*}
0 < n_* \leq n(s) \quad \forall s \in [-1,1].
\end{align*}
\item \label{ass:J} $J\in W^{1,1}_{\mathrm{loc}} (\mathbb{R}^d)$ such that
\begin{align*}
J(-z)=J(z), \quad a^*:= \sup_{x\in\Omega} \iO |J(x-y)| dy < \infty, \quad
 b:= \sup_{x\in\Omega} \iO |\nabla J(x-y)| dy < \infty.
\end{align*}
 \item \label{ass:P} $P \in C^0 ([-1,1])$ is non-negative, and there exist positive constants $k$ and $\eps_0$ such that
\begin{align*}
\sqrt{P(s)} \leq k m(s) \quad \forall s \in [-1,-1+\eps_0] \cup [1-\eps_0,1], \quad PF'\in C^0{([-1,1])}.
\end{align*}
\item \label{ass:ini} $\ph_0\in H$, $|\ph_0|\leq 1$ a.e.~in $\Omega$,
$M(\ph_0)\in L^1(\Omega)$, and $\s_0\in H$, where the entropy function $M\in C^2(-1,1)$ is defined by
$m(s)M''(s)=1$ for all $s\in(-1,1)$, and $M(0)=M'(0)=0$.
\end{enumerate}
For convenience, we will denote with $\laminfty$ and $P_\infty$ the uniform bound of $\lam$ and $P$, respectively.

\begin{rem}
{\upshape
We point out that, as a consequence of \eqref{ass:J}, we have that
\begin{align}
&\| J\star\ph\|_{p}\leq a^* \,\|\ph \|_{p},\quad
\| \nabla J\star\ph \|_{p}\leq b\,\| \ph \|_{p} \quad
\forall \ph \in L^p(\Omega)\label{conv-op}
\end{align}
and for all $1\leq p\leq \infty$. These estimates will be repeatedly
employed.
}
\end{rem}

\begin{rem}
{\upshape
Notice that, thanks to \eqref{ass:m:F} and to the definition of the entropy function $M$, the condition
$M(\ph_0)\in L^1(\Omega)$ in \eqref{ass:ini} implies $F(\ph_0)\in L^1(\Omega)$, see e.g.~\cite[Remark 1, p.~226]{FLR}.}
\end{rem}

\begin{rem}
{\upshape (Corrigendum for \cite{FLR}). At the beginning of \cite{FLR} the boundary conditions associated to system $(\mathrm{P})$ are given by \eqref{bc}, instead of $\dn \mu = 0$ as stated in \cite[(1.5)]{FLR}.
}
\end{rem}

\begin{rem}\label{alt_cond}
{\upshape
A careful look to the proof of \cite[Thm.~2.3]{FLR} shows that \eqref{ass:P} can be replaced by the following assumption, which is more general as far as the proliferation function $P$ is concerned
\begin{itemize}
\item[(A4*)] $P\in C^0([-1,1])$, $P\geq 0$, and $PF' \,,PM' \,,PM' F'\,,PF'' \in C^0([-1,1]).$
\end{itemize}
The advantage of this condition is that it allows us to include proliferation functions of the form
$P(s)=P_0(1-s^2)^\alpha\chi_{[-1,1]}(s)$, with $\alpha=1$,
once the mobility and potential are assumed as in \eqref{m:F:eg}, where $P_0$ denotes a non-negative constant.  Notice that, given \eqref{m:F:eg}, in order to satisfy \eqref{ass:P} we need $\alpha\geq 2$.
}
\end{rem}

Under the above assumptions, the existence of weak solutions can be obtained by employing a suitable approximation scheme that resembles the one introduced in \cite{EG}.  More precisely, an approximate problem is solved at first by suitably regularizing $F$, $P$ and $m$.  Then, uniform estimates with respect to the approximating parameter are derived which allow to pass to the limit by classical weak and strong compactness arguments.  The weak existence result for $(\mathrm{P})$ that can derived from \cite[Thm.~2.3]{FLR} is formulated as follows.

\begin{thm}\label{thm:weak}
Assume that \eqref{ass:m:F}-\eqref{ass:ini} are satisfied.  Then, there exists a weak solution $[\ph,\s]$ to $(\mathrm{P})$ in the following sense:
\begin{itemize}
\item
it enjoys the following regularity
\begin{align}\label{regweak}
\ph,\s &\in { H^1(0,T;V^*)\cap L^2(0,T;V)\subset C^0([0,T];H)}, \\[1mm]
\ph &\in L^\infty (Q), \quad |\ph(x,t)|\leq 1 \quad \text{a.e.~in }Q; \label{uniformbound}
\end{align}
\item for every $v,w\in V$ and almost every $t \in (0,T)$ we have that
\begin{align}
\notag
\<\dt \ph, v>_V + \iO \nabla \Lam (\ph) \cdot \nabla v & =  \iO Bm(\ph) (\nabla J\star \ph)\cdot  \nabla v \\
& \quad  +\iO P(\ph)(\s  - AF'(\ph)+ B J\star\ph ) v,
\label{W_first}\\
\label{W_second}
\<\dt \s, w >_V + \iO n(\ph)\nabla \s \cdot \nabla w & = - \iO P(\ph)(\s - AF'(\ph)+ B J\star\ph )w,
\end{align}
along with the initial conditions $\ph(0)=\ph_0$ and $\s(0)=\s_0$ in $H$.
\end{itemize}
Moreover, there exists a positive constant $K_1$
which depends only on $\Omega$, $T$, and on the data of the system such that
\begin{align*}
\|\ph\|_{{ H^1(0,T;V^*)\cap L^2(0,T;V)}} + \|\s \|_{{ H^1(0,T;V^*)\cap L^2(0,T;V)}} 	\leq K_1.
\end{align*}
\end{thm}
For continuous dependence on initial data (which also entails uniqueness of solutions) further assumptions are needed:
\begin{enumerate}[label=$(\mathrm{B \arabic*})$, ref = $\mathrm{B \arabic*}$, series=B]
\item \label{ass:cts:m:n} $m \in C^{0,1}([-1,1])$ and $n=1$.
\item \label{ass:cts:P:F} $P, PF' \in C^{0,1}([-1,1])$.
\end{enumerate}
\begin{thm}\label{thm:weakcts}
Suppose that \eqref{ass:m:F}-\eqref{ass:P} and \eqref{ass:cts:m:n}-\eqref{ass:cts:P:F} are satisfied.
Let $[\ph_i,\s_i]$, for $i=1,2$, be two
solutions to $(\mathrm{P})$ corresponding to initial data
$[\ph_{0,i},\s_{0,i}]$ satisfying \eqref{ass:ini}. Then, there exists a  positive constant $K_2$ which depends only on $\Omega$, $T$, and on the data of the system such that
\begin{align*}
& \| \ph_1-\ph_2\|_{L^\infty(0,T;V^*) \cap L^2(0,T;H)}+\|\s_1-\s_2\|_{L^\infty(0,T;V^*) \cap L^2(0,T;H)}
\\[1mm]
& \quad
\leq K_2 \Big (\|\ph_{0,1}-\ph_{0,2}\|_{V^*}+\|\s_{0,1}-\s_{0,2}\|_{V^*} \Big ).
\end{align*}
\end{thm}

\begin{rem}
{\upshape We point out that due to our choice of the non-local Ginzburg--Landau energy $\mathcal{F}$ in \eqref{nonloc:GL}, in the notation of \cite{FLR} we have $F_2 = 0$, $a(x) = 0$ and $F_1 = F$.  Hence, we can simplify several assumptions for well-posedness.
}
\end{rem}

\section{Strong well-posedness}\label{sec:strong}
Further regularity for the weak solution to $(\mathrm{P})$ can be established with a more regular convolution
kernel $J$.  For instance, the assumption $J\in W^{2,1}_{\rm loc}(\mathbb{R}^d)$ would be sufficient from an analytical point of view.  However, as
pointed out in \cite{FGG2}, this assumption excludes the physically relevant cases of Newtonian and
Bessel potential kernels. A way to \last{overcome} this issue is to assume that $J$ is {\em admissible} in the following sense:
\begin{defn}\label{defn:kernels}
A convolution kernel $J\in W^{1,1}_{\rm loc}(\mathbb{R}^d)$ is said to be admissible if it fulfils the following conditions:
\begin{itemize}
\item $J\in C^3(\mathbb{R}^d \setminus \{ 0 \})$.
\item $J$ is radially symmetric, i.e. $J(x)=\widetilde{J}(|x|)$ for a non-increasing function $\widetilde J$.
\item ${\widetilde J}'' (r) $ and ${\widetilde J}'(r)/r $ are monotone on $(0,r_0)$ for some $r_0>0$.
\item \last{There exists some positive constant $C_d$} such that $|D^3J(x)|\leq C_d|x|^{-d-1}$ \last{for all $x \neq 0$}.
\end{itemize}
\end{defn}
For strong well-posedness, we reinforce previous assumptions by assuming that:
\begin{enumerate}[label=$(\mathrm{C \arabic*})$, ref = $\mathrm{C \arabic*}$, series=C]
\item \label{ass:str:m:n} $m, n \in C^1([-1,1])$.
\item \label{ass:str:F:lam}  $F \in C^3 (-1,1)$ and $\lam \in C^1([-1,1])$.
\item \label{ass:str:J} $J\in W^{2,1}_{\rm loc}(\mathbb{R}^d)$ or $J$ is admissible in
the sense of Definition~\ref{defn:kernels}.
\item \label{ass:str:P} $P, PF' \in C^1([-1,1])$.
\end{enumerate}

\begin{thm}\label{thm:strong}
Assume that \eqref{ass:m:F}-\eqref{ass:ini} and \eqref{ass:str:m:n}-\eqref{ass:str:J} are satisfied for $d \in \{2,3\}$, and $\ph_0\in V$.  Then, there exists a weak solution $[\ph,\s]$ to $(\mathrm{P})$ which exhibits the additional regularity
\begin{align}\label{imp_reg1}
\ph  &\in H^1(0,T;H) \cap L^\infty(0,T;V) \cap L^2(0,T;W).
\end{align}
Furthermore, if $\s_0 \in V$, and assuming that $n=1$ when $d=3$,  it holds that
\begin{align}\label{imp_reg2}
\s \in H^1(0,T;H) \cap L^\infty(0,T;V) \cap L^2(0,T;W).
\end{align}
Lastly, for $d \in\{2, 3\}$ and $n=1$, if $\s_0 \in W$ with \last{$\dn \s_0 = 0$ on $\Gamma$} and \eqref{ass:str:P} also hold, then
\begin{align}\label{imp_reg3}
\s \in W^{1,\infty}(0,T;H) \cap H^1(0,T;V) \cap L^\infty(0,T;W).
\end{align}
\end{thm}
We are also able to prove a stronger regularity result.
\begin{thm}\label{thm:superstrong}
Let $\ph_0,\s_0 \in W$, with $\dn\s_0 = 0$ \last{ on $\Gamma$} and
\begin{align}
[\nabla \Lam(\ph_0)-Bm(\ph_0)(\nabla J \star \ph_0)]\cdot \bold{n} = 0 \text{ on } \Gamma.\label{nonlinBC}
\end{align}
Assume \eqref{ass:m:F}-\eqref{ass:ini} and \eqref{ass:str:m:n}-\eqref{ass:str:P} hold, and in addition
$\lambda= m F'' = \alpha_0$ is a constant and $n=1$ for the case $d = 3$.  Then, there exists a weak solution $[\ph,\s]$ to $(\mathrm{P})$
which, in addition to the regularities obtained by Theorem \ref{thm:strong}, exhibits the additional regularity
\begin{align}\label{imp_reg4}
\ph \in W^{1,\infty} (0,T;H) \cap H^1(0,T;V) \cap L^\infty(0,T;W), \quad \Lam(\ph) \in L^\infty(0,T;W),
\end{align}
and for the case $d = 2$, the regularity \eqref{imp_reg3} also hold for $\s$ without the previous restriction on the nutrient mobility $n$. Moreover, there exists a positive constant $K_3$ which depends only on
$\Omega$, $T$, $J$, and on the data of the system such that
\begin{align}
& \|\ph\|_{W^{1,\infty} (0,T;H) \cap H^1(0,T;V) \cap L^\infty(0,T;W)}
+\| \s \|_{W^{1,\infty} (0,T;H) \cap H^1(0,T;V) \cap L^\infty(0,T;W)} \leq K_3.
\label{reg_superstrong}
\end{align}
\end{thm}

We point out that $\lambda = m F''$ being a constant for the assumption of Theorem \ref{thm:superstrong} implies $\Lambda(s) = A \alpha_0 s$.  This does not take away the combination of degenerate mobility and singular potential from the non-local model.

Next, we present two improvements of the continuous dependence results of \cite{FLR} (see Theorem \ref{thm:weakcts}), where due to the improved regularity for $\ph$ we can consider a non-constant mobility $n(\ph)$ in the case $d = 2$.  This fact is new with respect to \cite{FLR}, where the regularity of the weak solution confines the analysis to the case of constant mobility $n = 1$.
The first improvement is a weak-strong continuous dependence result.

\begin{thm}\label{thm:cts1}
Assume that \eqref{ass:m:F}--\eqref{ass:P} and \eqref{ass:str:m:n}--\eqref{ass:str:P} are satisfied for $d\in \{2,3\}$. For $d=3$, suppose in addition
that $\lambda$ is a constant and $n=1$. Assume
that initial data $[\ph_{0,i}, \s_{0,i}]$, for $i = 1,2$, are given such that
$[\ph_{0,1},\s_{0,1}]\in V\times V$ and $[\ph_{0,2},\s_{0,2}]\in H\times H$ (with
$\ph_{0,1},\ph_{0,2}$ satisfying also \eqref{ass:ini}). Let $[\ph_1,\s_1]$, and $[\ph_2,\s_2]$ be the corresponding
solutions, given by Theorem \ref{thm:strong}, and by Theorem \ref{thm:weak}, respectively.
Then, there exists a positive constant $K_4$ which
depends only on $\Omega$, $T$, $J$, and on the data of the system
such that
\begin{equation}\label{cts1}
\begin{aligned}
 &\|\ph_1-\ph_2\|_{{ H^1(0,T;V^*)\cap L^2(0,T;V)}} + \| \s_1 - \s_2 \|_{{ H^1(0,T;V^*)\cap L^2(0,T;V)}} \\
& \quad \leq K_4 \Big ( \| \ph_{0,1} - \ph_{0,2} \| + \| \s_{0,1} - \s_{0,2} \|	\Big ).
\end{aligned}
\end{equation}
\end{thm}
Theorem \ref{thm:weakcts} and Theorem \ref{thm:cts1} entail uniqueness of the solution
to Problem (P). More precisely, we have the following
\begin{cor}\label{cor:uniq}
Assume that \eqref{ass:m:F}--\eqref{ass:ini} and \eqref{ass:str:m:n}--\eqref{ass:str:P} are satisfied for $d\in \{2,3\}$. For $d=3$ suppose in addition
that $\lambda$ is a constant and that $n=1$. Let the initial data satisfy one of the following
conditions: $(i)$ $\,[\ph_0,\s_0]\in H\times H$, if $n=1$; $(ii)$ $\,[\ph_0,\s_0]\in V\times V$.
Then, the solution to Problem $(P)$ given by Theorem \ref{thm:weak} and by Theorem \ref{thm:strong}, respectively,
is unique.
\end{cor}

In two spatial dimensions,
we can prove a stronger continuous dependence result.
To this aim we need the following conditions.
\begin{enumerate}[label=$(\mathrm{D \arabic*})$, ref = $\mathrm{D \arabic*}$, series=D]
\item \label{ass:cts:s:m:n}  $m, n \in C^{1,1}([-1,1])$.
\item \label{ass:cts:s:P:F} $P, PF' \in C^{1,1}([-1,1])$.
\end{enumerate}

\begin{thm}\label{thm:cts2}
Assume that $d=2$ and that \eqref{ass:m:F}--\eqref{ass:P}, \eqref{ass:str:m:n}--\eqref{ass:str:P}, \eqref{ass:cts:s:m:n}--\eqref{ass:cts:s:P:F} are satisfied.
Suppose in addition that
\begin{align*}
m,n \in C^2([-1,1]), \quad F\in C^4(-1,1), \quad \lam \in C^2([-1,1]).
\end{align*}
Assume
that initial data $[\ph_{0,i}, \s_{0,i}]$, for $i = 1,2$, are given such that
$[\ph_{0,i},\s_{0,i}]\in W\times W$, with $\dn \s_{0,i} = 0$ \last{on $\Gamma$} and with $\ph_{0,i}$
satisfying \eqref{nonlinBC} and \eqref{ass:ini} for $i=1,2$. Let $[\ph_1,\s_1]$ and $[\ph_2,\s_2]$ be the corresponding
strong solutions given by Theorem \ref{thm:superstrong}.
Then, there exists a  positive constant $K_5$
which depends only on $\Omega$, $T$, $J$, and on the data of the system
such that
\begin{align*}
&\|\ph_1-\ph_2\|_{H^1(0,T;H) \cap L^\infty(0,T;V) \cap L^2(0,T;W)} + \| \s_1 - \s_2 \|_{H^1(0,T;H) \cap L^\infty(0,T;V) \cap L^2(0,T;W)}  \\
& \quad \leq K_5 \Big ( \| \ph_{0,1} - \ph_{0,2} \|_V + \| \s_{0,1} - \s_{0,2} \|_V \Big ).
\end{align*}
\end{thm}

\subsection{Existence of strong solutions}
Let us first recall two useful lemmas.
\begin{lem}
\label{LEM_trace}
Let $f,g \in H^{1/2}(\Gamma)\cap L^\infty(\Gamma)$, \last{where $\Gamma:=\partial\Omega$, $\Omega\subset\mathbb{R}^d$,
 $d\in\{2,3\}$}.
Then, the product $fg \in H^{1/2}(\Gamma)\cap L^\infty(\Gamma)$ and
\begin{align} \label{traceineq}
\|fg\|_{H^{1/2}(\Gamma)} \leq \|f \|_{L^\infty(\Gamma)}\|g \|_{H^{1/2}(\Gamma)}	+ \|f \|_{H^{1/2}(\Gamma)} \|g \|_{L^\infty(\Gamma)}.
\end{align}
\end{lem}
We refer the reader to \cite[Chap.~IX, Sec.~18]{DB} for the proof and just recall that the space ${H^{1/2}(\Gamma)}$ is endowed with the following seminorm
\begin{align}\label{seminorm}
[f]_{H^{1/2}(\Gamma)}:=\Big ( \int_\Gamma\int_\Gamma \frac {|f(x)-f(y)|^2}{\last{|x-y|^d}} d\Gamma(x)d\Gamma(y) \Big )^{1/2},
\end{align}
where $d\Gamma$ stands for the surface measure on the boundary $\Gamma$.

Another advantage of considering admissible kernels in the
sense of Definition~\ref{defn:kernels} is the validity of the following result, \last{which holds in spatial dimensions $d \geq 2$ and} whose proof can be found in \cite[Lem.~2]{BRB}.
\begin{lem}\label{LEM_controlofdiv}
Assume that the kernel $J$ is admissible in the sense of Definition~\ref{defn:kernels}. Then, for every $p\in(1,\infty)$, there exists a positive constant $C_p$ such that
\begin{align}\label{controlofdiv}
\|\div(\nabla J\star \psi)\|_{L^p(\Omega)^{d \times d}}	\leq C_p \| \psi \|_{L^p(\Omega)} \qquad
\last{\forall \psi \in L^p(\Omega)}.
\end{align}
Moreover, $C_p={C_*} p$ if $p \in [2,\infty)$ and $C_p=C_* p/ (p-1)$ if $p \in (1,2)$ for a positive constant $C_*$ independent of $p$.
\end{lem}

\subsubsection{Proof of Theorem~\ref{thm:strong}}
We apply the argument outlined in \cite[Sec.~4]{FGGS}
to the system given by \eqref{eq_first}, \eqref{bc} only, where
$\s$ is taken as the weak solution given by Theorem \ref{thm:weak}. For fixed $\eps > 0$, we introduce the regular potential $F_{\eps}$, and  the functions
$m_{\eps}$, $P_{\eps}$ given by
\begin{align}
\{F_{\eps}''(s), m_\eps(s), P_\eps(s) \} := \begin{cases}
\{F''(1-\eps), m(1-\eps), P(1-\eps)\}& \, \, \, \, s \geq 1-\eps, \\
\{F''(s), m(s), P(s) \} & |s| \leq 1-\eps, \\
\{F''(\eps-1), m(\eps-1), P(\eps-1)\} & \, \, \, \, s \leq \eps-1,
\end{cases}\label{e18}
\end{align}
with $F_\eps(0) = F(0)$ and $F_\eps'(0) = F'(0)$. Then,
owing to \eqref{ass:m:F}, it is clear that the function $\lambda_\eps:=m_{\eps}F_\eps''$ satisfies the bounds
\begin{align}\label{e3}
0 < \alpha_0 \leq \lambda_\eps(s) \leq \max_{s \in [-1,1]} \lambda(s) =: \lambda_\infty \quad \forall s \in \mathbb{R}.
\end{align}
Moreover, we claim that there exist two constants $k_1,k_2>0$, independent
of $\eps$, such that the following bound holds
\begin{align}\label{e1}
&|P_{\eps}(s)\,F_{\eps}'(s)|\leq k_1+k_2|s|\qquad\forall s\in\mathbb{R}.
\end{align}
Indeed, for $s\geq 1-\eps$, we have that
\begin{align}\label{e2}
&|P_{\eps}(s)\,F_{\eps}^\prime(s)|\leq |(PF^\prime)(1-\eps)|+|(PF^{\prime\prime})(1-\eps)|(s-(1-\eps))
\leq k_1+k_2\,s,
\end{align}
where $k_1=\| PF^\prime \|_{L^\infty(-1,1)}$ (cf.~\eqref{ass:P} or (A4*)), and $k_2$ is given by $k_2=k\lambda_\infty\sqrt{P_\infty}$ due to \eqref{ass:P}  or by
$k_2=\| PF^{\prime\prime}\|_{L^\infty(-1,1)}$ in case (A4*) is assumed (cf.~Remark \ref{alt_cond}). For $s\leq -1+\eps$ the estimate is similar to \eqref{e2},
while for $|s|\leq 1-\eps$ we simply have $|P_{\eps}(s)\,F_{\eps}^\prime(s)|=|P(s)\,F^\prime(s)|\leq k_1$.  Hence, \eqref{e1} immediately follows.

Denoting by $\Q : \mathbb{R} \to \mathbb{R}$ the truncation function
\begin{align*}
\Q(s) = \max\{-1, \min\{s, 1\}\} \quad \forall s \in \mathbb{R},
\end{align*}
we first approximate \eqref{eq_first}, \eqref{bc} with the following system
\begin{align}
& \dt \ph - \Delta \Lam_{\eps} (\ph) = - B \, \div ( m_{\eps}(\ph)(\nabla J\star\Q(\ph))) - A (P_\eps F_\eps')(\Q(\ph))\label{eq_first_eps}\\[1mm]
\notag & \qquad \qquad \qquad  \qquad  +P_{\eps}(\ph)(\s + B J\star\Q(\ph)) \text{ in } \Omega, \\
& [\nabla \Lam_{\eps}(\ph) - B\, m_{\eps}(\ph)(\nabla J \star\Q(\ph))]\cdot \bold{n} = 0 \text{ on } \Gamma,\label{boundary_condition_eps}
\end{align}
where $\Lam_\eps(s) : = A \int_0^s \lambda_\eps(r) dr$ for every $s\in\mathbb{R}$.  We then prove that, for every $\eps>0$, system
\eqref{eq_first_eps}-\eqref{boundary_condition_eps}
admits a solution $\ph_\eps$ in the class \eqref{imp_reg1}.

To this aim, a time discretization scheme applied to
\eqref{eq_first_eps}-\eqref{boundary_condition_eps} is implemented as follows.
We first recall that $\s$ (which is now fixed) satisfies \last{$\sigma \in H^1(0,T;V^*) \cap L^2(0,T;V)\subset C^0([0,T];H)$.}
Now, fix $N \in \mathbb{N}$ and set the time step $\tau = T/N$.  For $k = 0, \dots, N-1$, given $\varphi_{k} \in V$ and $\sigma_k := \sigma(k \tau) \in H$, find $\varphi_{k+1} \in V$ solving
\begin{align}\label{eps_approx}
\notag - \tau \Delta \Lam_\eps(\varphi_{k+1}) + \varphi_{k+1} & = \varphi_{k} - \tau B \div \big ( m_\eps(\varphi_{k}) ( \nabla J \star \Q(\varphi_k)) \big ) - \tau A (P_\eps F_\eps')(\Q(\varphi_k)) \\[1mm]
& \quad + \tau P_\eps(\varphi_k) (\sigma_k  + B J \star \Q(\varphi_k)) 
\quad \text{ a.e.~in } \Omega,
\\[1mm]
\label{eps_bc} \nabla \Lam_\eps(\ph_{k+1}) \cdot \bold{n} & = B m_\eps(\ph_k)(\nabla J \star \Q(\ph_k)) \cdot \bold{n}  \quad \text{ a.e.~on } \Gamma.
\end{align}
Notice that $\sigma_k$ is well-defined on account of the regularity \last{$\sigma \in C^0([0,T];H)$.}
 The nonlinear operator $A_k : V \to V^*$ defined by
\begin{align*}
\<A_k \ph , \psi >_V := \tau ( \nabla \Lam_\eps(\ph), \nabla \psi) + (\ph, \psi) \quad \forall \ph, \psi \in V,
\end{align*}
is pseudomonotone and coercive on $V$ (cf.~\cite[Lems  2.31, 2.32 and 2.35]{Rou}), while
\begin{align*}
\langle g_k , \psi \rangle_V & := (\ph_k, \psi) + \tau B (m_\eps(\ph_k)(\nabla J \star \Q(\ph_k)), \nabla \psi) - \tau A ((P_\eps F_\eps')(\Q(\varphi_k)), \psi) \\[1mm]
& \quad + \tau (P_\eps(\ph_k)(\sigma_k + B J \star \Q(\ph_k)), \psi)  \quad \forall \psi \in V
\end{align*}
satisfies $g_k \in V^*$, thanks to the boundedness of $m_{\eps}$, $P_{\eps}$ and
to \eqref{e1}.  Therefore, \eqref{eps_approx}-\eqref{eps_bc}
can be written as an abstract problem
\begin{align*}
A_k \ph_{k+1}=g_k \text{ in } V^*
\end{align*}
and admits a solution $\ph_{k+1} \in V$ (see \cite[Thm.~2.6]{Rou}; cf.~also \cite{Bre}).
\last{By means of a comparison argument in \eqref{eps_approx}, together with elliptic regularity theory, we can easily infer that
we have also $\ph_{k+1} \in W$. Indeed, by noting that $m_\eps\in W^{1,\infty}(\mathbb{R})$, $P_\eps\in L^\infty(\mathbb{R})$
(cf. \eqref{e18}, \eqref{ass:P} and \eqref{ass:str:m:n}), by using \eqref{e1}, Lemma \ref{LEM_controlofdiv}
(which ensures that $\div  ( \nabla J \star \Q(\varphi_k))\in L^q(\Omega)$, for all $q\in(1,\infty)$),
and the fact that $\ph_k, \ph_{k+1} \in V$, $\sigma_k\in H$, we can deduce from \eqref{eps_approx} that $\Delta \Lam_\eps(\varphi_{k+1}) \in H$. Moreover, by taking advantage of \eqref{traceineq} and arguing as in \cite[Estimate (4.25)]{FGGS},
we can see that $\nabla \Lam_\eps(\varphi_{k+1}) \cdot \bold{n} \in H^{1/2}(\Gamma)$. Therefore,
elliptic regularity entails that $\Lam_\eps(\ph_{k+1}) \in W$, and this implies that $\nabla \Lam_\eps(\ph_{k+1}) =A
\lambda_\eps(\ph_{k+1}) \nabla \ph_{k+1} \in V$. On the other hand, by \eqref{ass:str:F:lam} we infer that $\nabla \lambda_\eps(\ph_{k+1}) = \lambda_\eps'(\ph_{k+1}) \nabla \ph_{k+1} \in L^4(\Omega)$. Thus $\nabla \ph_{k+1} = A^{-1}\lambda_\eps^{-1}(\ph_{k+1}) \nabla \Lam_\eps(\ph_{k+1}) \in V$, whence $\ph_{k+1} \in W$.}

Before we derive uniform discrete estimates, let us collect a useful elementary identity and several useful inequalities established in \cite{FGGS}, more precisely \eqref{FGGS1}, \eqref{FGGS2} and \eqref{FGGS3} below can be derived from equations (4.16), (4.25) and (4.27) of \cite{FGGS}, respectively. In the following $\delta$ denotes positive constants whose values are yet to be determined, while $C$ denotes positive constants independent of $N$, $\tau$ and $\eps$.  For $n \leq N-1$, it holds that
\begin{align}\label{elementary_id}
\sum_{k=0}^{n} (\ph_{k+1}-\ph_{k},\ph_{k+1}) =	\tfrac 12 \sum_{k=0}^{n} \|\ph_{k+1}-\ph_{k}\|^2 +\tfrac 12  \|\ph_{n+1}\|^2	-\tfrac 12  \|\ph_{0}\|^2,
\end{align}
and that
\begin{align}
\label{FGGS1} \tau B \left | \sum_{k=0}^{n} (m_\eps(\ph_k) (\nabla J \star \Q(\ph_{k})), \nabla \ph_{k+1})\right | \leq \delta \tau \sum_{k=0}^{n} \| \nabla \ph_{k+1} \|^2 + C, \\
\label{FGGS2} \delta \tau \sum_{k=0}^n \left \| \nabla \Lam_\eps(\ph_{k+1}) \cdot \bold{n} \right \|_{H^{1/2}(\Gamma)}^{2} \leq C \delta T + C \delta \tau \sum_{k=0}^n \| \ph_k \|_V^2, \\
\label{FGGS3} \delta \tau B \sum_{k=0}^n \| \div (m_\eps(\ph_k) (\nabla J \star \Q(\ph_k))) \|^2 \leq C \delta T + C \delta \tau \sum_{k=0}^n
\|\nabla \ph_k\|^2.
\end{align}
Now, integrating \eqref{eps_approx} yields
\begin{align*}
|(\ph_{k+1}, 1)| &  \leq |(\ph_{k},1)| + \tau |(P_\eps(\ph_k)(\sigma_k + B J \star \Q(\ph_k)) - A (P_\eps F_\eps')(\Q(\ph_k)), 1)| \\
&  \leq |(\ph_k, 1)| + C \tau
\end{align*}
on account of the boundedness of $(P_\eps F_\eps')(\Q(\ph_k))$ and $\|\sigma_k\| \leq \|\sigma\|_{L^\infty(0,T; H)} \leq C$.  In particular,
\begin{align*}
|(\ph_{k+1},1)| \leq |(\ph_0,1)| + C(k+1)\tau \leq |(\ph_0,1)| + CT \leq CT(1+\|\ph_0\|).
\end{align*}
Hence, by the Poincar\'e inequality
\begin{equation}\label{FGGS4}
\begin{aligned}
& \tau \sum_{k=0}^n |(P_\eps(\ph_k)(\sigma_k + B J \star \Q(\ph_k)) - A(P_\eps F_\eps')(\Q(\ph_k)), \ph_{k+1})| \\
& \quad \leq C \tau \sum_{k=0}^n \|\ph_{k+1} \| \leq C \tau \sum_{k=0}^n \| \ph_{k+1} - \mean{\ph_{k+1}}\| + |\mean{\ph_{k+1}}| \leq C + \delta \tau \sum_{k=0}^n \| \nabla \ph_{k+1} \|^2.
\end{aligned}
\end{equation}
Then, testing \eqref{eps_approx} with $\ph_{k+1}$, summing over $k$ from $k = 0$ to $k = n \leq N-1$,
employing the identity \eqref{elementary_id}, estimates \eqref{FGGS1}, \eqref{FGGS4} and \eqref{e3},
and choosing $\delta =   \alpha_0 A/4$ yields
\begin{align}\label{FGGS5}
& \tfrac{1}{2} \sum_{k=0}^n \| \ph_{k+1} - \ph_k\|^2 + \tfrac{1}{2} \| \ph_{n+1}\|^2 + \tfrac{\alpha_0A}{2} \tau \sum_{k=0}^n \| \nabla \ph_{k+1} \|^2 \leq \tfrac{1}{2} \| \ph_0 \|^2 + C.
\end{align}
An immediate consequence of \eqref{FGGS5} is
\begin{align}\label{FGGS6}
\tau \sum_{k=0}^n \|\ph_{k+1} \|^2 \leq C\tau \sum_{k=0}^n \| \nabla \ph_{k+1} \|^2 + C \tau \sum_{k=0}^n |\mean{\ph_{k+1}}|^2 \leq C(1+\|\ph_0\|^2).
\end{align}
Next, testing \eqref{eps_approx} with $\Lam_\eps(\ph_{k+1}) - \Lam_\eps(\ph_k)$, summing from $k =0 $ to $k = n$ and using \eqref{elementary_id} for $\nabla \Lam_\eps (\ph_k)$ yields
\begin{align*}
& \tfrac{\alpha_0A}{\tau} \sum_{k=0}^n  \| \ph_{k+1} - \ph_k \|^2 + \tfrac{1}{2} \| \nabla \Lam_\eps(\ph_{n+1}) \|^2 + \tfrac{1}{2} \sum_{k=0}^n \| \nabla \Lam_\eps(\ph_{k+1}) - \nabla \Lam_\eps(\ph_{k}) \|^2 \\
& \quad \leq B(m_\eps(\ph_{n+1}) (\nabla J \star \Q(\ph_{n+1})), \nabla \Lam_\eps(\ph_{n+1})) - B(m_\eps(\ph_0) (\nabla J \star \Q(\ph_0)), \nabla \Lam_\eps(\ph_0)) \\
& \qquad - B\sum_{k=0}^n (m_\eps(\ph_{k+1}) (\nabla J \star \Q(\ph_{k+1})) - m_\eps(\ph_k)(\nabla J \star \Q(\ph_k)), \nabla \Lam_\eps(\ph_{k+1})) \\
& \qquad + \sum_{k=0}^n (P_\eps(\ph_k) (\sigma_k + B J \star \Q(\ph_k)) - A (P_\eps F_\eps')(\Q(\ph_k)), \Lam_\eps(\ph_{k+1}) - \Lam_\eps(\ph_k)) \\
& \qquad +  \tfrac{1}{2} \| \nabla \Lam_\eps(\ph_0) \|^2.
\end{align*}
As in \cite[(4.19)-(4.20)]{FGGS}, the first and third terms on the right-hand side are bounded above by
\begin{align*}
\tfrac{1}{4} \| \nabla \Lam_\eps(\ph_{n+1})\|^2 + \tfrac{\alpha_0A}{4 \tau} \sum_{k=0}^n \| \ph_{k+1} - \ph_k\|^2 + C \tau \sum_{k=0}^n \| \nabla \Lam_\eps(\ph_{k+1}) \|^2 + C.
\end{align*}
Meanwhile, using that $\Lam_\eps(a) - \Lam_\eps(b) = A\int_b^a \lambda_\eps(s)ds$, the fourth term is bounded above as follows
\begin{align*}
& \left | \sum_{k=0}^n (P_\eps(\ph_k) (\sigma_k + B J \star \Q(\ph_k)) - A (P_\eps F_\eps')(\Q(\ph_k)), \Lam_\eps(\ph_{k+1}) - \Lam_\eps(\ph_k))  \right | \\
&\quad  \leq C \sum_{k=0}^n  \| \ph_{k+1} - \ph_k \| \leq \sum_{k=0}^n (C \tau + \tfrac{\alpha_0A}{4 \tau} \| \ph_{k+1} - \ph_k \|^2) \leq CT + \tfrac{\alpha_0A}{4 \tau} \sum_{k=0}^n \| \ph_{k+1} - \ph_k \|^2.
\end{align*}
Together with the fact that $\|\nabla \Lam_\eps(\ph_0) \| \leq A \lambda_\infty \| \nabla \ph_0 \|$, we find that
\begin{align*}
& \tfrac{\alpha_0A}{2 \tau} \sum_{k=0}^n  \| \ph_{k+1} - \ph_{k} \|^2 + \tfrac{1}{4} \| \nabla \Lam_\eps(\ph_{n+1}) \|^2 + \tfrac{1}{2} \sum_{k=0}^n \| \nabla \Lam_\eps(\ph_{k+1}) - \nabla \Lam_\eps(\ph_k) \|^2 \\
& \quad \leq C(1+\|\ph_0\|_V^2) + C \tau \sum_{k=0}^n \| \nabla \Lam_\eps(\ph_{k+1}) \|^2
\\
&\quad \leq C(1+\|\ph_0\|_V^2)
+C\tau\|\nabla\Lambda_{\eps}(\ph_{n+1})\|^2+C\tau\sum_{k=0}^{n-1} \| \nabla \Lam_\eps(\ph_{k+1}) \|^2.
\end{align*}
By taking $\tau$ small enough so that $C \tau < \frac{1}{4}$, and by applying the discrete Gronwall lemma, we deduce that
\begin{align}\label{FGGS7}
\tfrac{1}{\tau} \sum_{k=0}^n  \| \ph_{k+1} - \ph_k\|^2 + \| \nabla \Lam_\eps(\ph_{n+1}) \|^2 + \sum_{k=0}^n \| \nabla \Lam_\eps(\ph_{k+1})- \nabla \Lam_\eps(\ph_k) \|^2
\leq\mathbb{Q}(\|\ph_0\|_V).
\end{align}
From \eqref{FGGS5}-\eqref{FGGS6} we infer that
\begin{align*}
\tau \sum_{k=0}^n \| \Lam_\eps(\ph_{k+1}) \|_V^2 \leq C \tau \sum_{k=0}^n \| \ph_{k+1} \|_V^2
\leq C(1+\|\ph_0\|^2),
\end{align*}
while by \eqref{FGGS3} and \eqref{FGGS7} it holds that
\begin{equation}\label{FGGS7.5}
\begin{aligned}
& \tau \sum_{k=0}^n \| \Delta \Lam_\eps(\ph_{k+1}) \|^2 \\
& \quad \leq  \frac{C}{\tau} \sum_{k=0}^n \|\ph_{k+1} - \ph_k \|^2 + C \tau \sum_{k=0}^n   \| \div (m_\eps(\ph_k)(\nabla J \star \Q(\ph_k))) \|^2 \\
& \qquad + C \tau \sum_{k=0}^n \| P_\eps(\ph_k)(\sigma_k + B J \star \Q(\ph_k)) - A (P_\eps F_\eps')(\Q(\ph_k)) \|^2 \\
& \quad
\leq\mathbb{Q}(\|\ph_0\|_V).
\end{aligned}
\end{equation}
Employing elliptic regularity, \eqref{FGGS2} and the above estimates, we infer that
\begin{align}
	\notag
 \tau \sum_{k=0}^n \| \Lam_\eps(\ph_{k+1}) \|_{W}^2
&\leq C \tau \sum_{k=0}^n \big (\| \Lam_\eps(\ph_{k+1}) \|_V^2 + \| \nabla \Lam_\eps(\ph_{k+1}) \cdot \bold{n} \|_{H^{1/2}(\Gamma)}^2 + \| \Delta \Lam_\eps(\ph_{k+1}) \|^2 \big )\\
&  \leq\mathbb{Q}(\|\ph_0\|_V)
\label{FGGS8}
\end{align}
and consequently a similar argument to that used in \cite[Proof of Thm.~3.6, p.~695--696]{FGGS}
for the case $d=2$ yields
\begin{align}\label{FGGS9}
\tau \sum_{k=0}^n \| \ph_{k+1} \|_{W}^2 \leq\mathbb{Q}(\|\ph_0\|_V).
\end{align}
\last{Let us briefly sketch the arguments} \last{ for completeness. In two spatial dimensions}, \last{employing \eqref{interpol_abs_2d}, \eqref{FGGS7} and \eqref{FGGS8}, it is easy to see that
\begin{align*}
\tau \sum_{k=0}^{n} \| \nabla \Lam_\eps(\ph_{k+1}) \|_{\last{4}}^4 \leq C \tau \sum_{k=0}^{n} \| \nabla \Lam_\eps(\ph_{k+1}) \|^2 \| \Lam_\eps(\ph_{k+1}) \|_{W}^2 \leq \mathbb{Q}(\| \ph_0 \|_V).
\end{align*}
Thanks to the relations} \last{ $\nabla \Lam_\eps(\ph_{k+1})=A\lambda_\eps(\ph_{k+1}) \nabla \ph_{k+1}$,} \last{$\nabla \lambda_\eps(\ph_{k+1}) = \lambda_\eps'(\ph_{k+1}) \nabla \ph_{k+1}$, as well as the boundedness of $\lambda_\eps'$ from \eqref{ass:str:F:lam}, we have
\begin{align*}
\tau \sum_{k=0}^n \| \nabla \ph_{k+1} \|_{\last{4}}^4 + \tau \sum_{k=0}^{n} \| \nabla \lambda_\eps(\ph_{k+1}) \|_{\last{4}}^4 \leq C \tau \sum_{k=0}^n \| \nabla \Lam_\eps(\ph_{k+1}) \|_{\last{4}}^4 \leq \mathbb{Q}(\| \ph_0 \|_V).
\end{align*}
Then, from \eqref{FGGS8} and the identity (for $i,j \in \{1,2\}$)
\begin{align*}
A \partial_{ij}^2 \ph_{k+1} = \frac{1}{\lambda_\eps(\ph_{k+1})} \partial_{ij}^2 \Lam_\eps(\ph_{k+1}) - \frac{1}{\lambda^2(\ph_{k+1})} \partial_i \lambda_\eps(\ph_{k+1}) \partial_j \Lam_\eps(\ph_{k+1}),
\end{align*}
we infer \eqref{FGGS9}.
}
We mention that in the case $d = 3$, a Moser--Alikakos iteration is used in \cite[Proof of Thm.~6.1, p.~718--719]{FGGS} to first establish $\| \ph_{k+1} \|_{L^\infty(\Omega)} \leq C ( \| \ph_0 \|_{L^\infty(\Omega)})$ for all $k = 0, \dots, n$ with $n < N-1$, which is then used to show \eqref{FGGS9}.  In our setting we have the additional source term $S_k := P_\eps(\ph_k)(\s_k + B J \star \Q(\ph_k)) - A (P_\eps F_\eps')(\Q(\ph_k))$ in \eqref{eps_approx}, and at present we cannot directly replicate the Moser--Alikakos argument as $\s_k$ is currently not bounded in $L^\infty(0,T;L^\infty(\Omega))$.

However, let us claim that the control given by \eqref{FGGS9} can be achieved also for $d = 3$
assuming $n=1$ with similar arguments provided
we consider the nutrient variable $\sigma_k$ to possess the stronger regularity
pointed out by Theorem \ref{thm:superstrong} which in turn would give
us $\s_k \in L^\infty(0,T;L^\infty(\Omega))$ (cf.~Remark \ref{Moser}).
In fact, it turns out that the stronger regularity for
the nutrient, in the case $n=1$, just requires further assumption on
$\s_0$ and that we still have
$\ph \in H^1(0,T; H)\cap L^\infty(0,T; V)
\cap L^2(0,T; W)$ also for the case $d=3$ without the assumption $n=1$
(see \eqref{sergio:1}--\eqref{sergio:4} below).

Next, we introduce the functions $\hat \ph_N$, $\ph^+_N$ and $\ph^-_N$ as interpolations of $\{\ph_n\}_{0 \leq n \leq N}$ in the following way:
\begin{align*}
\begin{cases}
\hat \ph_N(t) := \gamma_n(t) \ph_{n} + (1-\gamma_n(t)) \ph_{n+1}, \quad \gamma_n(t) = n+1 - t/ \tau, \\
\ph^+_N(t) := \ph_{n+1}, \\
\ph^-_N(t) := \ph_{n},
\end{cases}
\end{align*}
for $n\tau < t < (n+1) \tau$ and $n = 0,\dots, N-1$.  Then, the estimates \eqref{FGGS5}-\eqref{FGGS9} imply (for both $d \in \{2,3\}$)
\begin{align*}
& \| (\hat \ph_N)' \|_{L^2(0,T; H)}^2  + \| \ph^+_N \|_{L^\infty(0,T;V)}^2 + \| \ph^-_N \|_{L^\infty(0,T; V)}^2 +\|\hat{\ph}_N\|_{L^\infty(0,T;V)}^2 \\[1mm]
& \quad\quad + \tau^{-1} \big ( \|  \hat \ph_N - \ph^+_N \|_{L^2(0,T;H)}^2 + \| \hat \ph_N - \ph^-_N \|_{L^2(0,T;H)}^2 \big ) \leq C,\\[1mm]
& \| \Lam_\eps(\ph^+_N) \|_{L^2(0,T;W)}^2 + \| \Lam_\eps( \ph^+_N) \|_{L^\infty(0,T; V)}^2 \leq C,
\end{align*}
and in the case $d = 2$ we additionally have,
\begin{align*}
 \| \ph^+_N \|_{L^2(0,T;W)}^2 \leq C.
\end{align*}
Now, \eqref{eps_approx}-\eqref{eps_bc} can be rewritten in terms of the interpolating functions
$\hat\ph_N$, $\ph^+_N$, $\ph^-_N$ as follows
\begin{align}
\partial_t\hat\ph_N & =\Delta\Lambda_{\eps}(\ph^+_N)-B\div(m_{\eps}(\ph^-_N)(\nabla J\star\Q(\ph^-_N)))\nonumber\\[1mm]
&\quad+P_{\eps}(\ph^-_N)(\widetilde{\sigma}_N+B J\star\Q(\ph^-_N)) - A(P_\eps F_\eps')(\Q(\ph^-_N)) \quad \text{ a.e.~in } \Omega,
\label{e4}\\[1mm]
\nabla\Lambda_{\eps}(\ph^+_N)\cdot\bold{n}& =B m_{\eps}(\ph^-_N)(\nabla J\star\Q(\ph^-_N))\cdot\bold{n} \quad \text{ a.e.~on } \Gamma,
\label{e5}
\end{align}
where $\widetilde{\sigma}_N$ is defined by $\widetilde{\sigma}_N(t):=\sigma_k$, for $k\tau<t<(k+1)\tau$, $k=0,\dots, N-1$.
Notice that the following strong convergence holds
\begin{align}
&\widetilde{\sigma}_N\to\sigma \text{ in }L^\infty(0,T;H).\label{e6}
\end{align}

Arguing as in \cite{FGGS}, for fixed $\eps>0$, we first pass to the limit as $N \to \infty$
in \eqref{e4}-\eqref{e5} (on account of \eqref{e6} as well), to deduce the existence of a function $\ph_\eps \in H^1(0,T; H) \cap L^\infty(0,T;V) $ with $\Lam_\eps(\ph_\eps) \in L^\infty(0,T; V) \cap L^2(0,T;W)$, and if $d = 2$ also $\ph_\eps \in L^2(0,T;W)$, satisfying
\begin{align*}
\dt \ph_\eps & = \Delta \Lam_\eps(\ph_\eps) - B\div ( m_\eps(\ph_\eps) ( \nabla J \star \Q(\ph_\eps))) \\
& \quad + P_\eps(\ph_\eps)(\sigma + B J \star \Q(\ph_\eps)) - A(P_\eps F_\eps')(\Q(\ph_\eps))
\end{align*}
a.e.~in $Q$, and the boundary condition
\begin{align*}
\nabla \Lam_\eps(\ph_\eps) \cdot \bold{n} = Bm_\eps(\ph_\eps) (\nabla J \star \Q(\ph_\eps)) \cdot \bold{n} \quad \text{ a.e.~on } \Sigma.
\end{align*}
Then, as all the estimates are uniform in $\eps$, we pass to the limit $\eps \to 0$ to obtain a limit function
\begin{align*}
& \ph \in H^1(0,T;H) \cap L^\infty(0,T;V), \, \quad \ph \in L^2(0,T;W) \text{ if } d = 2, \\
& \text{with } \Lam(\ph) \in L^\infty(0,T;V) \cap L^2(0,T;W), \quad  |\ph| \leq 1 \text{ a.e.~in } Q,
\end{align*}
and satisfies
\begin{align*}
\dt \ph = \Delta \Lam(\ph) - B\div ( m(\ph) ( \nabla J \star \ph)) + P(\ph)(\sigma - A F'(\ph) + B J \star \ph) \quad \text{ a.e.~in } Q, \\[1mm]
\nabla \Lam(\ph) \cdot \bold{n} = B m(\ph) (\nabla J \star \ph) \cdot \bold{n} \quad \text{ a.e.~on } \Sigma.
\end{align*}
We refer the reader to \cite{FLR} on the arguments to pass to the limit for the term $(P_\eps F_\eps')(\Q(\ph_\eps))$ and also to deduce that $|\ph| \leq 1$ a.e.~in $Q$ which removes the truncation function $\Q$.

Now, we claim that at the continuous level, by exploiting the above regularity for $\ph$ and
$\Lambda(\ph)$, we can derive the $\ph \in L^2(0,T;W)$-regularity also for $d=3$.
In this direction, let us notice that, for a.e.~$t\in (0,T)$, we have
$\Lambda(\ph(t)) \in W$ so that
\begin{align}\label{sergio:1}
	\partial_i \Lambda(\ph(t))= A \lambda(\ph(t)) \partial_i \ph(t) \in V \quad \hbox{for $i \in \{1,2,3\}.$}
\end{align}
Moreover, using \eqref{ass:m:F} and \eqref{ass:str:F:lam} we realize that $\lambda^{-1}(\ph(t)) \in V$.  Hence, for $i \in \{1,2,3\}$ and a.e.~$t \in (0,T)$, we have
\begin{align}\label{sergio:2}
	\partial_i \ph(t)
	=   \underbrace{\lambda^{-1}(\ph(t))}_{\in V}
	 \underbrace{\lambda(\ph(t)) \partial_i \ph(t)}_{\in V \subset L^6(\Omega)}
	\in W^{1,3/2}(\Omega) \subset L^3(\Omega)
\end{align}
which entails $\nabla \ph(t) \in L^3(\Omega)$ and also $\lambda^{-1}(\ph(t)) \in  W^{1,3}(\Omega)$.  In turn, by \eqref{ass:m:F}, we have
\begin{align}\label{sergio:3}
		\partial_i \ph(t)
	=   \underbrace{\lambda^{-1}(\ph(t))}_{\in  W^{1,3}(\Omega)\cap L^\infty(\Omega)}
	 \underbrace{\lambda(\ph(t)) \partial_i \ph(t)}_{\in V \subset L^6(\Omega)}
	\in H^{1}(\Omega) .
\end{align}
Therefore, we get $\partial_i \ph(t) \in V$ for $i \in \{1,2,3\}$, and so $\ph(t) \in W$ for a.e.~$t \in (0,T)$. Rigorously,
for a.e.~$t \in (0,T)$ and for $i,j \in \{1,2,3\}$, it holds that
\begin{align}
	A \partial_{i k}^2 \ph(t)  =
	\frac 1 {\lam(\ph(t) )}\partial_{i k}^2  \Lam (\ph(t) )
	- \frac {\lam'(\ph(t) )}{\last{A} \lam^3(\ph(t) )} \partial_{i}  \Lam (\ph(t) )\partial_{k}  \Lam (\ph(t) ).
	\label{sergio:first}
\end{align}
On the other hand, \last{we employ the following special case of the Gagliardo--Nirenberg inequality \eqref{e19},
which holds in both two and three spatial dimensions}
\begin{align}
	\label{sergio:second}
	\|\nabla \Lam(\ph(t) )\|_4
	\leq
	C \|\Lam(\ph(t) )\|^{1/2}_{\last{\infty}}
		\|\Lam(\ph(t) )\|^{1/2}_{\last{W}}
\end{align}
\last{(take $p=4$, $j=1$, $q=\infty$, $m=2$, $r=2$, $d\in\{2,3\}$, yielding $\alpha=1/2$ in \eqref{e19}).}
Therefore, the regularity $\Lam (\ph) \in L^2(0,T; W)$, the bound $|\ph|\leq 1$,
\eqref{sergio:first} and \eqref{sergio:second} imply that also for $d=3$ we have
\begin{align}\label{sergio:4}
	\ph \in L^2(0,T; W).
\end{align}

Hence, the proof of \eqref{imp_reg1} is concluded. The next step is now to prove \eqref{imp_reg2}, provided that
the stated additional assumptions on $\sigma_0$ and on the nutrient mobility $n$ are satisfied.
This is achieved by relying on the improved regularity for $\varphi$ given by \eqref{imp_reg1}. We proceed by means of formal estimates, which can be made rigorous by applying a Galerkin approximation to \eqref{eq_second}.

We multiply \eqref{eq_second} by $-\Delta \s$ (which is a valid test function in the Galerkin approximation) and integrate
over $\Omega$ and by parts to obtain that
\begin{align}\label{e7}
\frac 12 \frac d {dt} \|\nabla \s\|^2 + \|\sqrt{n(\ph)}\Delta \s\|^2
& = - \iO n'(\ph) \nabla \ph \cdot \nabla \s \Delta \s \notag \\ &  \quad + \iO P(\ph)(\s  - AF'(\ph)+ B J\star\ph )\Delta \s,
\end{align}
where the terms on the right-hand side are denoted by $I_1$ and $I_2$. By means of \eqref{ass:P} along with the Young inequality and the previous estimate, we obtain
that
\begin{align}\label{e8}
	|I_2|	\leq \delta \|\Delta \s\|^2 + C_\delta (1 + \| \s\|^2),
\end{align}
for a positive $\delta$ yet to be determined.
As for the first term, if $d = 2$, we use the boundedness of $n'$, the Gagliardo--Nirenberg inequality \eqref{interpol_abs_2d}
and the elliptic estimate
\begin{align}\label{el:s}
\|\s\|_{W} \leq C \big ( \|\Delta \s\| + \|\s\| \big )
\end{align}
to find that
\begin{align}\label{e9}
|I_1|	&\leq C \|\nabla \ph\|_4 \|\nabla \s\|_4 \|\Delta \s\|\notag\\
& \leq \delta \|\Delta \s\|^2 + C_\delta \|\nabla \ph\| \|\ph\|_{W} \|\nabla \s\| \|\s\|_{W} \notag\\
& \leq \delta \|\Delta \s\|^2 	+ C_\delta \|\nabla \ph\| \| \ph\|_{W} \|\nabla \s\| \big (\|\Delta\s\|+ \| \s\|\big ) \notag \\
&\leq 2\delta  \|\Delta \s\|^2 + C_\delta \|\nabla \ph\|^2 \|\ph\|_{W}^2 \|\nabla \s\| \big (\|\s\| + \|\nabla \s\| \big ).
\end{align}
It is worth noting that in the last term, accounting for the above estimates, we have that $ \|\nabla \ph\|\in L^\infty(0,T)$, $\| \ph \|_{W}^2 \in L^1(0,T)$, and $\| \s \| \in L^\infty(0,T)$ due to \eqref{regweak}.  Therefore, we insert \eqref{e8}-\eqref{e9} into \eqref{e7}, fix $\delta\in(0,1)$ small enough, use \eqref{ass:n} and apply Gronwall's lemma to get the bound
\begin{align}
\| \nabla \s \|_{L^\infty(0,T;H)} + \|\Delta \s\|_{L^2(0,T;H)} \leq \mathbb{Q}(\| \ph_0 \|_V,\| \s_0 \|_V). \label{est_s}
\end{align}
Then, the elliptic estimate \eqref{el:s} yields the $L^2(0,T;W)$-bound for $\s$.  Let us point out that if $d=3$, the argument leading to \eqref{est_s} does not work. The obstacle is the estimate of the term $I_1$ on account of the known regularity for $\ph$ and for $\s$.  On the other hand, if $n=1$, we simply have $I_1=0$, while the estimate for $I_2$ remains unchanged. So, the case $d=3$ and $n=1$ easily follows, and in particular it does not rely on the $L^2(0,T;W)$-regularity for $\ph$.

An estimate for $\dt \s$ can also be deduced, by means of a comparison argument in \eqref{eq_second}.
Indeed, we can write \eqref{eq_second} in the form
\begin{align}\label{e10}
&\dt \s=n(\ph)\Delta\s+n^\prime(\ph)\nabla\ph\cdot\nabla\s
- P(\ph)(\s - AF'(\ph)+ B J\star\ph ),
\end{align}
and estimate the second term on the right-hand side (present only in the case $d = 2$) as
\begin{align*}
\| n'(\ph)\nabla\ph\cdot\nabla\sigma \| &\leq C \|\nabla\ph\|_4 \|\nabla\s\|_4
\leq C \| \nabla\ph \|^{1/2} \|\ph\|_{W}^{1/2}\|\nabla\s\|^{1/2} \| \s\|_{W}^{1/2}\\
&\leq \mathbb{Q}(\|\ph_0\|_V,\|\s_0\|_V) \|\ph\|_{W}^{1/2} \| \s\|_{W}^{1/2},
\end{align*}
which, on account of the $L^2(0,T;W)$-regularity for $\ph$ and for $\s$, entails that
\begin{align*}
&\| n^\prime(\ph)\nabla\ph\cdot\nabla\sigma\|_{L^2(0,T;H)}\leq \mathbb{Q}(\|\ph_0\|_V,\|\s_0\|_V).
\end{align*}
Hence, from this estimate, \eqref{est_s}, \eqref{ass:n}, \eqref{ass:str:m:n} and \eqref{e10} we have for $d = 2$ the estimate
\begin{align}
&\| \dt \s\|_{L^2(0,T;H)}\leq  \mathbb{Q}(\|\ph_0\|_V,\|\s_0\|_V).\label{e12}
\end{align}
In the case $d=3$,
since we are assuming $n=1$, the second term on the right-hand side of \eqref{e10} is not present, and the $L^2(0,T;H)$-regularity for $\dt \s$ proceeds with a similar argument.
Therefore, \eqref{imp_reg2} is proven.

It then remains to establish the improved regularity \eqref{imp_reg3} for $\s$ for the case $d\in\{2, 3\}$ and $n=1$. To this aim, we formally differentiate \eqref{eq_second} in time and test with $\dt \s$, which can be made rigorous by returning to a Galerkin approximation of \eqref{eq_second} for $\s$ treating $\ph$ possessing the above improved regularity $H^1(0,T;H)$ as given data, and differentiating the Galerkin approximation in time.  Since $n = 1$ and $\s_0 \in W$, the latter implies $\dt \s_0:= \Delta \s_0 - P(\ph_0)(\s_0 - A F'(\ph_0) + B J \star \ph_0) \in H$,
and we obtain, after integration over $\Omega$,
\begin{align*}
\frac{1}{2} \frac{d}{dt} \| \dt \s \|^2 + \int_\Omega P(\ph) |\dt \s|^2 + |\nabla \dt \s |^2
	 &= - \int_\Omega  \big( P'(\ph) \s \dt \ph - B P(\ph) (J \star \dt \ph) \big) \dt \s \\
& \quad - \int_\Omega \big( B P'(\ph) \dt \ph (J \star \ph)  - A(P F')'(\ph) \dt \ph \big) \dt \s \\
&  \leq C \big ( 1 + \| \s \|_{W}^2 \big ) \| \dt \s \|^2  + C \| \dt \ph \|^2,
\end{align*}
\last{where we have used \eqref{ass:str:P}} \last{for the terms involving $P^\prime$ and $(PF')'$}. By Gronwall's lemma we have $\s \in W^{1,\infty}(0,T;H) \cap H^1(0,T;V)$.  Then, by a comparison of terms in \eqref{eq_second} with $n = 1$, we see that $\Delta \s \in L^\infty(0,T;H)$ and so by elliptic regularity it holds that
\begin{align*}
\s \in L^\infty(0,T;W).
\end{align*}
The proof of Theorem \ref{thm:strong} is complete. \qed

\begin{rem}\label{Moser}{\upshape
We point out that estimate \eqref{FGGS9}, which yields the control
of the discretized solutions $\varphi_N^+$ in $L^2(0,T;W)$ can be recovered
also for the case $d=3$, \last{provided \eqref{ass:str:P} is satisfied which leads to the improved regularity \eqref{imp_reg3} for $\s$}. Indeed, this allows to reproduce the Moser--Alikakos type argument of \cite[Proof of Thm 6.1]{FGGS}
which establishes the crucial bound $\| \ph_{k+1} \|_{L^\infty(\Omega)} \leq C ( \| \ph_0 \|_{L^\infty(\Omega)})$ for all $k = 0, \dots, n$ with $n < N-1$. Let us just sketch the main points of this argument.
We return to the time-discrete problem \eqref{eps_approx}-\eqref{eps_bc} taking the above improved regularity
for $\s$ into account, which implies that the source term $S_k = P_\eps(\ph_k)(\s_k + BJ \star \Q(\ph_k)) - A (P_\eps F_\eps')(\Q(\ph_k))$ is now uniformly bounded in $L^{\infty}(0,T;L^\infty(\Omega))$.  Appealing now to the Moser--Alikakos computation in \cite[p.~718--719]{FGGS}, which involves testing \eqref{eps_approx} with $\ph_{k+1}^{p_j - 1}$ where $p_j := 2^j$, integrating over $\Omega$ and summing the resulting identity over $k$, for $k = 0, \dots, n$ with $0 \leq n \leq N-1$, we obtain
\begin{equation}\label{FGGS6.13}
\begin{aligned}
& \frac{1}{p_j} \int_\Omega \ph_{n+1}^{p_j} + \frac{4 \alpha_0}{p_j p_j'} \tau \sum_{k=0}^n \int_\Omega |\nabla ( \ph_{k+1}^{p_j/2} )|^2 \\
& \quad \leq \frac{1}{p_j} \int_\Omega \ph_0^{p_j} + \tau \sum_{k=0}^{n} \Big [ B(m_\eps(\ph_k) (\nabla J \star \Q(\ph_k)), \nabla (\ph_{k+1}^{p_j - 1})) + (S_k, \ph_{k+1}^{p_j - 1}) \Big ],
\end{aligned}
\end{equation}
where $p_j' = \frac{p_j}{p_j - 1}$ is the conjugate of $p_j$.  The new element in the analysis is the last term on the right-hand side which, owing to the uniform boundedness of $S_k$ and Young's inequality, can be handled as
\begin{align*}
|(S_k, \ph_{k+1}^{p_j - 1})| \leq C \| \ph_{k+1}^{p_j - 1} \|_1 \leq C \| \ph_{k+1}^{p_j / 2} \|^{2\frac{p_j - 1}{p_j}} \leq C \frac{\delta}{p_j'} \int_\Omega | \ph_{k+1}^{p_j/2}|^2 + \frac{C \delta^{1-p_j}}{p_j},
\end{align*}
with a constant $C$ independent of $\delta$, the index $j$ and $N$.  Choosing $\delta = p_j$ and noting that $p_j^{-p_j} \to 0$ as $j \to \infty$ we infer that
\begin{align*}
|(S_k, \ph_{k+1}^{p_j-1})| \leq C\frac{p_j}{p_j'} \int_\Omega |\ph_{k+1}^{p_j/2}|^2 + C
\end{align*}
with a constant $C$ independent of index $j$ and $N$.  The convolution term in \eqref{FGGS6.13} can be handled as in \cite[(6.14)]{FGGS}
so that, after multiplying \eqref{FGGS6.13} by $p_j$ and estimating the convolution term and the source term,
we arrive at the following inequality
\begin{align*}
\int_\Omega | \ph_{n+1}^{p_j/2}|^2 + \frac{\alpha_0}{p_j'} \tau \sum_{k=0}^n \int_\Omega | \nabla ( \ph_{k+1}^{p_j/2})|^2 \leq \int_\Omega |\ph_0^{p_j/2}|^2 + C p_j^2 \tau \sum_{k=0}^n \int_\Omega |\ph_{k+1}^{p_j/2}|^2 + C p_j\tau,
\end{align*}
which is exactly \cite[(6.15)]{FGGS}.  Then, we may argue as in \cite{FGGS} to deduce that \eqref{FGGS9} also holds in the case $d = 3$ with a modified constant $\mathbb{Q}(\| \ph_0\|_V, \| \s_0 \|_{W})$ on the right-hand side. }
\end{rem}

\subsubsection{Proof of Theorem~\ref{thm:superstrong}}
We first consider the case $d = 2$ and return to the time-discrete problem \eqref{eps_approx}-\eqref{eps_bc}, taking the improved regularity $\s \in H^1(0,T;H)$ into account. Consider the differences between steps $k$ and $k-1$,  test the resulting identity by $\pdt \ph_{k+1} := \tau^{-1}( \ph_{k+1} - \ph_k)$ and sum over $k = 1, \dots, n$ with $n \leq N-1$, to get
\begin{align}
& \sum_{k=1}^n (\pdt \ph_{k+1} - \pdt \ph_{k}, \pdt \ph_{k+1})\notag \\
& = - \sum_{k=1}^n \big ( \nabla \Lam_\eps(\ph_{k+1}) - \nabla \Lam_\eps(\ph_k), \nabla \pdt \ph_{k+1} \big ) + B \sum_{k=1}^n \tau \big ( \pdt (m_\eps(\ph_k)(\nabla J \star \Q(\ph_k))),  \nabla \pdt \ph_{k+1} \big ) \notag\\
& \qquad + \tau  \sum_{k=1}^n \big ( \pdt [(P_\eps(\ph_k)(\sigma_k + B J \star \Q(\ph_k))) - A (P_\eps F_\eps')(\Q(\ph_k))], \pdt \ph_{k+1} \big ),\label{e11}
\end{align}
where we use the notation $\pdt f(\ph_k) = \tau^{-1} (f(\ph_k) - f(\ph_{k-1}))$.  \last{Let us collect some useful inequalities established in \cite{FGGS}, more precisely \eqref{FGGS:4.71} and \eqref{FGGS:4.73} below can be derived from equations (4.71) and (4.73) of \cite{FGGS}, respectively:
\begin{align}
& \tau \sum_{k=1}^n \Big ( \nabla \pdt \Lam_\eps(\ph_{k+1}), \nabla \pdt \ph_{k+1} \Big ) \notag  \\
& \quad \geq \tfrac{\alpha_0 \tau}{4} \sum_{k=1}^n \| \nabla \pdt \ph_{k+1} \|^2 - C \tau \sum_{k=1}^n \| \Lam(\ph_k) \|_{W}^2 \| \pdt \ph_{k+1} \|^2 - C \tau \sum_{k=1}^n \| \pdt \ph_{k+1} \|^2, \label{FGGS:4.71} \\
& \sum_{k=1}^n (\pdt (m_\eps(\ph_k) (\nabla J \star \Q(\ph_k))), \nabla \pdt \ph_{k+1} \big ) \leq \delta \tau \sum_{k=1}^n \| \pdt \nabla \ph_{k+1} \|^2 + C_\delta \tau \sum_{k=1}^n \| \pdt \ph_k \|^2,\label{FGGS:4.73}
\end{align}
where $\delta$ denotes a positive constant whose value is yet to be determined.}  Firstly, from \eqref{FGGS:4.73} with $\delta = \frac{\alpha_0 A}{16 B}$ and \eqref{FGGS7} it holds that
\begin{align}\label{e15}
B \tau \sum_{k=1}^n  \big ( \pdt (m_\eps(\ph_k)(\nabla J \star \Q(\ph_k))), \nabla \pdt \ph_{k+1} \big ) \leq \tfrac{\alpha_0 A \tau}{16} \sum_{k=1}^n \| \nabla \pdt \ph_{k+1} \|^2 + C.
\end{align}
Next, we can easily check that the following uniform (w.r.t.~$\eps$) Lipschitz continuity property of $P_\eps F_\eps'$
holds
\begin{align}\label{e13}
&|P_\eps(s_2) F_\eps'(s_2)-P_\eps(s_1) F_\eps'(s_1)|\leq L|s_2-s_1|\qquad\forall s_1,s_2\in\mathbb{R},
\end{align}
where the positive constant $L$ is independent of $\eps>0$ and is given by $L=k \lambda_\infty\sqrt{P_\infty} +\|(PF^\prime)^\prime\|_{L^\infty(-1,1)}$, in case \eqref{ass:P} holds, or by $L=\| PF^{\prime\prime}\|_{L^\infty(-1,1)}+\|(PF^\prime)^\prime\|_{L^\infty(-1,1)}$, in case (A4*) holds. Indeed, let us consider, e.g., the case $|s_1|<1-\eps$ and $s_2\geq 1-\eps$ (the other cases can be handled similarly), we then have
\begin{align*}
&|P_\eps(s_2) F_\eps'(s_2)-P_\eps(s_1) F_\eps'(s_1)|\notag\\[1mm]
& \quad =|P(1-\eps)F^\prime(1-\eps)
+P(1-\eps)F^{\prime\prime}(1-\eps)(s_2-(1-\eps))-P(s_1)F^\prime(s_1)|\notag\\[1mm]
&\quad \leq \| PF^{\prime\prime}\|_{L^\infty(-1,1)}(s_2-(1-\eps))+\|(PF^\prime)^\prime\|_{L^\infty(-1,1)}
((1-\eps)-s_1)\leq L(s_2-s_1),
\end{align*}
where we have assumed (A4*). By employing \eqref{e13}, the Lipschitz continuity of $\Q$, the Gagliardo--Nirenberg inequality, estimate \eqref{FGGS7}, as well as $\sigma \in C^0([0,T];L^4(\Omega))$ (which is due to the application of Aubin--Lions lemma with the compact embedding $V \subset \subset L^4(\Omega)$),
and also the identity $\pdt(f_k g_k)=f_k\pdt g_k+g_{k-1}\pdt f_k$, we get
\begin{equation}\label{e14}
\begin{aligned}
& \tau \sum_{k=1}^n  \big ( \pdt [(P_\eps(\ph_k)(\sigma_k  + B J \star \Q(\ph_k))) - A (P_\eps F_\eps')(\Q(\ph_k))], \pdt \ph_{k+1} \big ) \\
& \quad \leq C \tau \sum_{k=1}^n \| \pdt \ph_{k+1} \| \| \pdt \ph_{k} \| + \tau \sum_{k=1}^n \big (P_\eps(\ph_{k}) \pdt \sigma_{k} + \sigma_{k-1} \pdt P_\eps(\ph_{k}), \pdt \ph_{k+1} \big )  \\
& \quad \leq C \tau \sum_{k=1}^n \| \pdt \ph_{k+1} \| \big (\| \pdt \ph_k \| + \| \pdt \s_k \| \big ) + C \tau \sum_{k=1}^n \| \s_{k-1} \|_4 \| \pdt \ph_k \| \| \pdt \ph_{k+1} \|_4 \\
& \quad \leq C \tau \sum_{k=1}^n \Big(\| \pdt \ph_{k+1} \|^2 + \|\pdt \ph_k \|^2 + \| \pdt \s_k \|^2 + \| \pdt \ph_{k+1} \| \| \nabla \pdt \ph_{k+1} \| \Big )
\\
& \quad \leq C + C \tau \sum_{k=1}^n \| \pdt \s_k \|^2 + \tfrac{\alpha_0 A \tau}{16} \sum_{k=1}^n \| \nabla \pdt \ph_{k+1} \|^2.
\end{aligned}
\end{equation}
As far as the first term on the right hand side of \eqref{e11} is concerned, by \eqref{FGGS:4.71} and \eqref{FGGS7}, it holds that
\begin{equation}\label{FGGS4.71ours}
\begin{aligned}
& \sum_{k=1}^n \big ( \nabla \Lam_\eps(\ph_{k+1}) - \nabla \Lam_\eps(\ph_k), \nabla\pdt \ph_{k+1} \big ) \\
& \quad \geq \tfrac{\alpha_0 A\tau}{4} \sum_{k=1}^n \| \nabla \pdt \ph_{k+1} \|^2 - C \tau \sum_{k=1}^n \| \Lam_\eps(\ph_{k}) \|_{W}^2 \| \pdt \ph_{k+1} \|^2 - C,
\end{aligned}
\end{equation}
and on account of \eqref{FGGS7}, of estimate $\| \nabla \ph_{k} \| = \| \nabla \Lam_\eps(\ph_k) / \lambda_\eps(\ph_k) \|$, and of \eqref{e1},  we obtain
\begin{equation}\label{e16}
\begin{aligned}
\tau\|\Lambda_{\eps}(\ph_k)\|_{W}^2 & \leq C \tau \big ( \| \Delta \Lam_\eps(\ph_k) \|^2 + \| \Lam_{\eps}(\ph_k) \|_V^2 + \| \nabla \Lam_{\eps}(\ph_k) \cdot \bold{n} \|_{H^{1/2}(\Gamma)}^2 \big ) \\
& \leq C \tau \big ( \| \pdt \ph_{k} \|^2 + \| \div (m_\eps(\ph_{k-1})( \nabla J \star \Q(\ph_{k-1})) \|^2 \\
& \; \quad + \| P_\eps(\ph_{k-1})(\s_{k-1}  + B J \star \Q(\ph_{k-1})) - A (P_\eps F_\eps')(\Q(\ph_{k-1}))\|^2  \\
& \; \quad + \| \ph_{k} \|_V^2 + \| m_\eps(\ph_{k-1}) \|_{L^\infty(\Gamma)}^2 \|  J \star \Q(\ph_{k-1}) \|_{W}^2 \\
& \; \quad + \| m_\eps(\ph_{k-1}) \|_{V}^2 \| (\nabla J \star \Q(\ph_{k-1}))  \cdot \bold{n} \|_{L^\infty(\Gamma)}^2 \big) \\
& \leq \frac{C}{\tau}\|\ph_k-\ph_{k-1}\|^2+C\tau,
\end{aligned}
\end{equation}
cf.~\cite[Proof of Thm.~3.6, Step 3, p.~704]{FGGS}. Inserting estimates \eqref{e15}, \eqref{e14}, \eqref{FGGS4.71ours} and \eqref{e16}  into
\eqref{e11}
we find that
\begin{equation}\label{FGGS10}
\begin{aligned}
& \tfrac{1}{2} \| \pdt \ph_{n+1} \|^2 + \tfrac{1}{2} \sum_{k=1}^n \| \pdt \ph_{k+1} - \pdt \ph_{k} \|^2 + \tfrac{\alpha_0 A\tau}{8} \sum_{k=1}^n \| \nabla \pdt \ph_{k+1} \|^2 \\
&\quad \leq C\tau \sum_{k=1}^n \| \Lam_\eps(\ph_{k}) \|_{W}^2  \| \pdt \ph_{k+1} \|^2 + C \tau \sum_{k=1}^n \| \pdt \sigma_{k} \|^2 + \tfrac{1}{2} \| \pdt \ph_1 \|^2 + C \\
& \quad \leq \frac{C}{\tau} \sum_{k=0}^{n-1} \|\ph_{k+2}-\ph_{k+1}\|^2\Big\|\frac{\ph_{k+1}-\ph_k}{\tau}\Big\|^2
+\frac{C}{\tau}\sum_{k=1}^{n}\|\ph_{k+1}-\ph_k\|^2
\\
& \qquad +C \tau \sum_{k=1}^n \| \pdt \sigma_{k} \|^2 + \tfrac{1}{2} \| \pdt \ph_1 \|^2 + C.
\end{aligned}
\end{equation}
A similar argument as in \cite[(4.83)]{FGGS} shows that
\begin{align*}
\| \pdt \ph_1 \| & \leq C \| \Delta \Lam_\eps(\ph_0) \| + C \| \div (m_\eps(\ph_0) (\nabla J \star \Q(\ph_0)) \| \\
& \quad + C \|P_\eps(\ph_0) (\sigma_0  + B J \star \ph_0) - A(P_\eps F_\eps')(\Q(\ph_0))\|  \\
& \leq C \big ( 1 +  \| \sigma_0 \| + \| \ph_0 \|_{W} \big ).
\end{align*}
By the Cauchy--Schwarz inequality, the fundamental theorem of calculus, and recalling estimate \eqref{e12}, we have that
\begin{align*}
\tau \sum_{k=1}^n \| \pdt \sigma_{k} \|^2 \leq \| \dt \sigma \|_{L^2(0,T;H)}^2 \leq \mathbb{Q}(\|\ph_0\|_V,\|\sigma_0\|_V).
\end{align*}
Hence, by employing this last estimate in \eqref{FGGS10} the discrete Gronwall lemma entails that
\begin{equation}\label{FGGS11}
\begin{aligned}
\| \pdt \ph_{n+1} \|^2 + \tau \sum_{k=1}^n \| \nabla \pdt \ph_{k+1} \|^2  + \sum_{k=1}^n \| \pdt \ph_{k+1} - \pdt \ph_{k} \|^2 \leq \mathbb{Q}(\|\ph_0\|_{W},
\|\sigma_0\|_V).
\end{aligned}
\end{equation}
This implies that the interpolation functions $\hat \ph_N$ introduced in the proof of Theorem \ref{thm:strong} now also satisfy
\begin{align*}
\| (\hat \ph_N)' \|_{L^\infty(0,T;H)\cap L^2(0,T;V)}
	\leq
\mathbb{Q}(\|\ph_0\|_{W},\|\sigma_0\|_V).
\end{align*}
\last{Returning to \eqref{e16}, which can be written as
\begin{align*}
\| \Lam_\eps(\ph_{k+1}) \|_{W} \leq C \big ( 1 + \| \pdt \ph_{k+1} \| \big )
\end{align*}
} we infer that
\begin{align}\label{inftyH2}
\| \Lam_\eps(\ph^+_N) \|_{L^\infty(0,T;W)} \leq C, \quad \| \ph^+_N \|_{L^\infty(0,T;W)} \leq C.
\end{align}
Hence, after passing to the limit as $N \to \infty$ and then as $\eps \to 0$, we deduce \eqref{imp_reg4}.

We now turn to the regularity \eqref{imp_reg3} for $\s$ in the case $d = 2$, \last{where in contrast to Theorem \ref{thm:strong}, we now allow for a non-constant mobility $n(\ph)$}.  By formally differentiating \eqref{eq_second} in time and testing with $\dt \sigma$ we obtain
\begin{align*}
& \frac{1}{2} \frac{d}{dt} \| \dt \s \|^2 + \iO n(\ph) |\nabla \dt \s |^2 + \iO P(\ph) |\dt \s |^2 \\
& \quad =
- \iO n'(\ph) \dt \ph \nabla \s \cdot \nabla \dt \s
-\iO \big(P'(\ph)\s  \dt \ph + BP(\ph) (J \star \dt \ph) \big)\dt \s \\
& \qquad
- \iO   \big(B P'(\ph) \dt \ph (J \star \ph) - A(P F')'(\ph) \dt \ph \big) \dt \s,
\end{align*}
where the right-hand side is bounded above by
\begin{align*}
& \delta \| \nabla \dt \s \|^2 + C \| \dt \ph \|_{4}^2 \| \nabla \s \|_{4}^2 + C\| \s \|_4 \| \dt \ph \|_4 \| \dt \s \| + C \| \dt \ph \|^2 + C \| \dt \s \|^2.
\end{align*}
Then, by the Gagliardo--Nirenberg inequality and the fact that $\dt \ph \in L^\infty(0,T; H) \cap L^2(0,T; V)$ and $\s \in L^\infty(0,T; V) \cap L^2(0,T;W)$ \last{due to \eqref{imp_reg2}},
\begin{align*}
\| \dt \ph \|_{4}^2 \| \nabla \s \|_{4}^2 &\leq C \| \dt \ph \| \| \dt \ph \|_V \| \nabla \s \| \| \s \|_{W}
\notag\\
& \leq \mathbb{Q}(\|\ph_0\|_V,\|\s_0\|_V) \| \dt \ph \|_V^2
 + \mathbb{Q}(\|\ph_0\|_{W},\|\s_0\|_V) \| \s \|_{W}^2, \\[1mm]
\| \s \|_4 \| \dt \ph \|_4 \| \dt \s \| &\leq \mathbb{Q}(\|\ph_0\|_V,\|\s_0\|_V)
 \| \dt \ph \|_V^2 + C \| \dt \s \|^2.
\end{align*}
Hence, choosing $\delta$ sufficiently small, and using \eqref{ass:n}, \eqref{ass:str:m:n} and \eqref{e12}, we infer by Gronwall's lemma that
\begin{align*}
\s \in W^{1,\infty}(0,T;H) \cap H^1(0,T;V).
\end{align*}
By comparison of terms in \eqref{eq_second}, it is easy to see that $\Delta \s \in L^\infty(0,T;H)$, and by a classical elliptic regularity argument, we also
deduce $\s \in L^\infty(0,T;W)$.

Let us briefly point out the modifications to the arguments for attaining the regularity \eqref{imp_reg4} in the case $d = 3$.  Thanks to Theorem \ref{thm:strong} we have $\s \in L^\infty(0,T;W)$, and so in \eqref{e14} the term $(\s_{k-1} \pdt P_\eps(\ph_k), \pdt \ph_{k+1})$ can be controlled by $C \| \pdt \ph_k \| \| \pdt \ph_{k+1} \|$, so that \eqref{e14} remains valid.  Moreover, under the assumption $\lambda = m F'' = \alpha_0$ is a constant, we have
\begin{equation}\label{FGGS4.71bis}
\begin{aligned}
& \sum_{k=1}^n \big ( \nabla \Lam_\eps(\ph_{k+1}) - \nabla \Lam_\eps(\ph_k), \nabla \pdt \ph_{k+1} \big ) \\
& \quad = \sum_{k=1}^n A \alpha_0 \big ( \nabla (\ph_{k+1} - \ph_k), \nabla \pdt \ph_{k+1} \big ) = A \alpha_0 \tau \sum_{k=1}^n \| \nabla \pdt \ph_{k+1} \|^2
\end{aligned}
\end{equation}
replacing \eqref{FGGS4.71ours}.  As \eqref{e16} remains unchanged, we obtain instead of \eqref{FGGS10} the inequality
\begin{equation}\label{e17}
\begin{aligned}
& \tfrac{1}{2} \| \pdt \ph_{n+1} \|^2 + \tfrac{1}{2} \sum_{k=1}^n \| \pdt \ph_{k+1} - \pdt \ph_k \|^2 + \tfrac{7\alpha_0 A \tau}{8} \sum_{k=1}^n \| \nabla \pdt \ph_{k+1} \|^2 \\& \quad \leq C \tau \sum_{k=1}^n \| \pdt \s_k \|^2 + \tfrac{1}{2} \| \pdt \ph_1 \|^2 + C \leq \mathbb{Q}(\|\ph_0\|_{W},\|\sigma_0\|_V),
\end{aligned}
\end{equation}
and the improved regularity \eqref{imp_reg4} follows along similar arguments as in the case $d = 2$. \qed

\subsubsection{Proof of Theorem~\ref{thm:cts1}}
Let us denote for convenience
\begin{align*}
\ph:= \ph_1-\ph_2, \quad	\s:= \s_1-\s_2,	\quad \hat{\Lam}:=\Lam_1-\Lam_2,
	\quad S := S_1 - S_2,
\end{align*}
where $S_i = P(\ph_i)(\s_i - A F'(\ph_i) + B J \star \ph_i)$,
$\Lam_i =\Lam(\ph_i)$, $n_i = n(\ph_i)$ and $m_i = m(\ph_i)$ for $i \in \{1,2\}$. First let us consider the case $d = 2$.  We take the difference of \eqref{eq_second} tested by $\s$, which yields, after integration over $\Omega$ and for some $\delta > 0$ to be fixed later,
\begin{equation}\label{Cts:1}
\begin{aligned}
\frac{1}{2} \frac{d}{dt} \| \s \|^2 + \iO n_2 |\nabla \s |^2 & = - \iO S \s + \iO(n_1 - n_2) \nabla \s_1 \cdot \nabla \s
\\
& \leq - \iO S \s  + C_\delta \| \nabla \s_1 \|_{4}^2 \| \ph \|_{4}^2 + \delta \| \nabla \s \|^2 \\
& \leq - \iO S \s  +C_\delta\|\nabla\s_1\|\|\s_1\|_{W}\|\ph\|\|\ph\|_V+ \delta \| \nabla \s \|^2\\
& \leq - \iO S \s +C_\delta\|\nabla\s_1\|\|\s_1\|_{W}\|\ph\|^2 \\[1mm]
& \quad
+\delta (\|\nabla\ph\|^2 + \| \nabla \s \|^2)
+C_\delta\|\nabla\s_1\|^2\|\s_1\|_{W}^2\|\ph\|^2
\end{aligned}
\end{equation}
on account of the fact that $\s_1 \in L^\infty(0,T;V)\cap L^2(0,T;W)$ and of the Gagliardo--Nirenberg inequality.
Moreover, we have used \eqref{ass:str:m:n} and the bound $|\varphi_i|\leq 1$, for $i=1,2$.
 Next, testing the difference of \eqref{eq_first} with $\ph$ yields
\begin{equation}\label{Cts:2}
\begin{aligned}
\frac{1}{2} \frac{d}{dt} \| \ph \|^2 + \iO \nabla \hat{\Lam} \cdot \nabla \ph = \iO B ( m_1(\nabla J \star \ph_1) - m_2(\nabla J \star \ph_2)) \cdot \nabla \ph + \iO S \ph.
\end{aligned}
\end{equation}
Thanks to the formulae
\begin{align}
\label{cts:Lam}
\nabla \hat \Lam =
\nabla \Lam(\ph_1) - \nabla \Lam(\ph_2) &= A (\lam(\ph_1)-\lam(\ph_2))\nabla\ph_1	+ A  \lam(\ph_2) \nabla \ph,
\\
m_1 (\nabla J \star \ph_1) - m_2 (\nabla J \star \ph_2) & = (m_1 - m_2) (\nabla J \star \ph_1)	+ m_2 (\nabla J \star \ph),
\end{align}
and
\begin{align}
\notag
S & = (P(\ph_1)- P(\ph_2)) (\s_1  + B J\star\ph_1) + P(\ph_2)(\s  + B J\star\ph)
\\
& \quad - A(P(\ph_1)F'(\ph_1)- P(\ph_2)F'(\ph_2)),
\label{defn:S}
\end{align}
from \eqref{Cts:2} and \eqref{ass:m:F} we get
\begin{align*}
& \frac{1}{2} \frac{d}{dt} \| \ph \|^2 + A \alpha_0 \| \nabla \ph \|^2 \\
& \quad \leq C \| \ph \|_4 \| \nabla \ph_1 \|_4 \| \nabla \ph \| + C \| \ph \| \| \nabla \ph \| + C(1+\Vert\s_1\Vert_W) \| \ph \|^2 + C \| \s \| \| \ph \| \\
& \quad \leq  2\delta \| \nabla \ph \|^2+C_\delta\|\ph\|\|\ph\|_V\|\nabla\ph_1\|\|\ph_1\|_{W}
+  C(1+\Vert\s_1\Vert_W)\| \ph \|^2 + C \| \s \|^2 \\
&\quad \leq 3\delta \| \nabla \ph \|^2+ C_\delta(1+\|\ph_1\|_{W}^2+\Vert\s_1\Vert_W) \| \ph \|^2 + C \| \s \|^2\, ,
\end{align*}
on account of the fact that $|\ph_1|\leq 1 \,\, a.e.$ in $Q$, $\ph_1,\s_1 \in L^\infty(0,T;V)\cap L^2(0,T;W)$ and the Gagliardo--Nirenberg inequality.
Moreover, in the first inequality we have used \eqref{ass:str:m:n}, \eqref{ass:str:F:lam} and \eqref{ass:str:P}. Estimating the term $S \s$ in \eqref{Cts:1} in a similar fashion using the Lipschitz continuity of $P$ and $PF'$, and adding the result to the above inequality yields
\begin{align}\label{cts:diff:1}
\frac{d}{dt} \Big ( \| \ph \|^2 + \| \s \|^2 \Big ) + \| \nabla \ph \|^2 + \| \nabla \s \|^2
\leq C (1+\|\ph_1\|_{W}^2+\|\s_1\|_{W}^2)\| \ph \|^2 + C \| \s \|^2\,,
\end{align}
on account of \eqref{ass:n}, and after choosing $\delta$ sufficiently small.  Then, Gronwall's lemma applied to \eqref{cts:diff:1} gives the $L^\infty(0,T;H) \cap L^2(0,T;V)$-estimate of \eqref{cts1} for the case $d = 2$.  For the $H^1(0,T;V^*)$-estimate we obtain from the difference of \eqref{eq_first} and \eqref{eq_second}
\begin{align*}
\langle \dt \ph, v \rangle_V & = - \int_\Omega \big( \nabla \hat{\Lambda} - B(m_1 (\nabla J \star \ph_1) - m_2(\nabla J \star \ph_2))\big) \cdot \nabla v
+\iO S v, \\
\langle \dt \s, w \rangle_V & = -\int_\Omega ((n_1 - n_2) \nabla \s_1 + n_2 \nabla \s) \cdot \nabla w  -\iO  S w,
\end{align*}
for any $v, w \in V$.  In light of the above estimates, as well as the calculations in \eqref{Cts:1}, we readily infer
\begin{align*}
\| \dt \ph \|_{L^2(0,T;V^*)} & \leq C\| \ph \|_{L^\infty(0,T;H)}^{1/2} \| \ph \|_{L^2(0,T;V)}^{1/2} + C\| \ph \|_{L^2(0,T;V)} + C \| \s \|_{L^2(0,T;H)}, \\
\| \dt \s \|_{L^2(0,T;V^*)} & \leq C \| \s \|_{L^\infty(0,T;H)}^{1/2} \| \s\|_{L^2(0,T;V)}^{1/2} + C \| \s \|_{L^2(0,T;V)} + C \| \ph \|_{L^2(0,T;H)},
\end{align*}
where the constants $C$ depend on the $L^4(0,T;L^4(\Omega))$-norms of $\nabla \ph_1$ and $\nabla \s_1$.  Applying the $L^\infty(0,T;H) \cap L^2(0,T;V)$-estimate of \eqref{cts1} then finishes the proof.

For $d = 3$, we note that the term $(n_1 - n_2) \nabla \s_1 \cdot \nabla \s $ in \eqref{Cts:1},
and the term $(\lambda(\ph_1) - \lambda(\ph_2)) \nabla \ph_1 \cdot \nabla \ph$ in \eqref{Cts:2}
both vanish, since, in this case, $n$ and $\lambda$ are assumed to be constant.  Hence, it is immediate to check that we can infer an analogous differential inequality to \eqref{cts:diff:1}
(the term $\|\ph_1\|_{W}^2$ will no longer appear on the right hand side). This allows us, by means
of Gronwall's lemma, to recover the assertion \eqref{cts1}. \qed

\subsubsection{Proof of Theorem~\ref{thm:cts2}}
\last{We remind the reader that the spatial dimension is $d = 2$}. Using the notation introduced at the beginning of the
proof of Theorem~\ref{thm:cts1}, we
test the difference of \eqref{eq_second} with $\dt \s = \dt \s_1 - \dt \s_2$ to obtain
\begin{align*}
& \frac{1}{2} \frac{d}{dt} \iO n_2 |\nabla \s|^2  + \| \dt \s \|^2 \\
& \quad = \iO (n_2 - n_1) \nabla \s_1 \cdot \nabla \dt \s + \frac{1}{2}\iO n'_2 \dt \ph_2 |\nabla \s|^2 - \iO S \dt \s \\
& \quad =\frac{1}{2}
 \iO n_2' \dt \ph_2 |\nabla \s|^2 - \iO S \dt \s - \iO(n_2 - n_1) \Delta \s_1 \dt \s \\
& \qquad - \iO \big ( (n'_2 - n'_1) \nabla \ph_2 \cdot \nabla \s_1 - n'_1 \nabla \ph \cdot \nabla \s_1 \big ) \dt \s,
\end{align*}
where we have set $n^\prime_i:=n^\prime(\varphi_i)$, $i=1,2$.
In light of the regularity of the solution stated in Theorem \ref{thm:superstrong},
of the boundedness of $n_i$ for $i=1,2$, and of conditions \eqref{ass:P}, \eqref{ass:str:P} and \eqref{ass:cts:s:m:n},
 we see that the right-hand side can be bounded by
\begin{align*}
& C \| \dt \ph_2 \| \| \nabla \s \|_4^2  + C \| \ph \|^2 + C \| \s \|^2 + C \| \Delta \s_1 \|^2 \| \ph \|_{W}^2 \\
& \qquad + \| \ph \|_{6}^2 \| \nabla \ph_2 \|_6^2 \| \nabla \s_1 \|_6^2 + C\| \nabla \ph \|_4^2  \| \nabla \s_1 \|_4^2 + \tfrac{1}{2} \| \dt \s \|^2 \\
& \quad \leq \last{C \| \nabla \s \|_4^2 + C \| \s \|^2 + C_1 \| \ph \|_W^2 + \tfrac{1}{2} \| \dt \s \|^2,}
\end{align*}
\last{for some positive constant $C_1 > 0$}. Applying the Gagliardo--Nirenberg inequality \eqref{interpol_abs_2d} to the term $\| \nabla \s \|_4^2$, for some $\eps_1 > 0$ to be determined later we find that
\begin{equation}\label{Hcts:1}
\begin{aligned}
\frac 12\frac{d}{dt} \iO n_2 |\nabla \s|^2 + \frac{1}{2} \| \dt \s \|^2 & \leq C \|n_2^{1/2}\nabla \s \|^2 + C \| \s \|^2 + C_1 \| \ph \|_{W}^2 + \eps_1 \| \s\|_{W}^2.
\end{aligned}
\end{equation}
To close this estimate we now derive estimates on $\| \ph \|_{W}$ and $\| \s \|_{W}$.  By taking the scalar product of \eqref{cts:Lam} with $\nabla \ph$, on account of \eqref{ass:m:F} and of \eqref{ass:str:F:lam}, we find the estimate
\begin{align*}
A\alpha_0 \| \nabla \ph \|^2 & \leq \| \nabla \hat \Lam \| \| \nabla \ph \|
+ C \| \ph \|_4 \| \nabla \ph_1 \|_4\| \nabla \ph \| \\
& \leq \| \nabla \hat \Lam \| \| \nabla \ph \| + C\|\ph \|^{1/2}\| \nabla \ph \|^{3/2} + C \| \nabla \ph\| \| \ph \|
\end{align*}
in light of the regularity $\ph_1 \in L^\infty(0,T;W)$.  This yields
\begin{align}\label{Cts:3}
\| \nabla \ph \| \leq C\big ( \| \nabla \hat \Lam \| + \| \ph \| \big ).
\end{align}
Then, from the identities \eqref{prop_lam}, again on account of \eqref{ass:str:F:lam} we find that
\begin{equation}\label{Cts:4}
\begin{aligned}
(\dt \ph, \dt \hat \Lam) & \geq A \alpha_0 \| \dt \ph \|^2 +A ((\lambda(\ph_1) - \lambda(\ph_2)) \dt \ph_1, \dt \ph) \\
& \geq (A \alpha_0 - \delta) \| \dt \ph \|^2 - C_\delta
 \| \dt \ph_1 \|_4^2 ( \| \ph \| \| \nabla \ph \| + \| \ph \|^2 ) \\
& \geq (A \alpha_0 - \delta) \| \dt \ph \|^2 - C_\delta
 \| \dt \ph_1 \|_4^2 ( \| \ph \|^2 +  \| \nabla \hat \Lam \|^2 ).
\end{aligned}
\end{equation}
Furthermore, from \eqref{defn:S}, we see that
\begin{align*}
\dt S & = ((P'(\ph_1) - P'(\ph_2)) \dt \ph_1 + P'(\ph_2) \dt \ph ) ( \s_1 + B J\star \ph_1 ) \\
& \quad + (P(\ph_1) - P(\ph_2))(\dt \s_1+B J \star \dt \ph_1) + P'(\ph_2) \dt \ph_2 (\s + B J \star \ph) \\
& \quad + P(\ph_2)(\dt \s + B J \star \dt \ph) -A \big ( (PF')'(\ph_1) - (PF')'(\ph_2) \big ) \dt \ph_1 -A(PF')'(\ph_2) \dt \ph,
\end{align*}
and so in light of the regularity stated in Theorem \ref{thm:superstrong}, of the inequality
\begin{align}\label{cts:Lam:2}
|\hat \Lam| \leq A \lambda_\infty |\ph |,
\end{align}
and of condition \eqref{ass:cts:s:P:F},
we find that for $\delta > 0$ and $\eps_2 > 0$ to be determined later
\begin{equation}\label{Cts:5}
\begin{aligned}
|(\hat \Lam, \dt S)| & \leq C \| \dt \ph_1 \| \| \ph \|_4^2 + C \| \dt \ph \| \| \ph \| + C ( \| \dt \s_1 \| + \| \dt \ph_1 \|) \| \ph \|_4^2 \\
& \quad + C \| \dt \ph_2 \| \| \ph \|_4 ( \| \s \|_4 + \| \ph \|_4)
+ C (\| \dt \s \| + \| \dt \ph \|) \| \ph \| \\
& \quad + C \|\dt \ph_1 \| \| \ph \|_4^2 + C \| \dt \ph \| \| \ph \| \\
& \leq \delta \| \dt \ph \|^2 + \eps_2 \| \dt \s \|^2 + C_{\delta,\eps_2} \| \ph \|_V^2
+ C_\delta \| \s \|_V^2.
\end{aligned}
\end{equation}
Now, set
\begin{align*}
\widetilde M & := B((m'_1 - m'_2) \dt \ph_1 + m'_2 \dt \ph)( \nabla J \star \ph_1 )\\
& \quad + B( m_1 - m_2)(\nabla J \star \dt \ph_1) +B m'_2 \dt \ph_2
 (\nabla J \star \ph ) + B m_2 ( \nabla J \star \dt \ph),
\end{align*}
where $m^\prime_i:=m^\prime(\varphi_i)$, for $i=1,2$. Then, on account of \eqref{ass:cts:s:m:n}
and of \eqref{Cts:3} as well, we have that
\begin{equation}\label{Cts:6}
\begin{aligned}
|(\nabla \hat \Lam, \widetilde M)| & \leq C \| \nabla \hat  \Lam \| \big ( \| \dt \ph_1 \|_4 \| \ph \|_4 + \| \dt \ph_2 \|_4 \| \ph \|_4 + \| \dt \ph \| \big ) \\
& \leq \delta \| \dt \ph \|^2 +C_\delta \| \nabla \hat \Lam \|^2
+ C_\delta \big ( \| \dt \ph_1 \|_4^2
 + \| \dt \ph_2 \|_4^2 \big ) \big ( \| \ph \| \| \nabla \ph \| +  \| \ph \|^2 \big )\\
& \leq \delta \| \dt \ph \|^2 + C_\delta \| \nabla \hat \Lam \|^2
+ C_\delta \big ( 1  + \| \dt \ph_1 \|_4^4 + \| \dt \ph_2 \|_4^4 \big) \| \ph \|^2.
\end{aligned}
\end{equation}
Then, testing the difference of \eqref{eq_first} with $\dt \hat \Lam$ yields
\begin{align*}
\frac{1}{2} \frac{d}{dt} \| \nabla \hat \Lam \|^2 + ( \dt \ph, \dt \hat \Lam)
 = B( \nabla  \dt \hat \Lam, (m_1- m_2) ( \nabla J \star \ph_1) + m_2( \nabla J \star \ph) )
 +(S, \dt \hat \Lam) ,
\end{align*}
which can be written as
\begin{align}\label{Cts:7}
& \frac{d}{dt} \Psi +  (\dt \ph, \dt \hat \Lam) + (\nabla \hat \Lam, \widetilde M) + (\hat \Lam, \dt S) = 0,
\end{align}
where \last{by \eqref{cts:Lam:2}}
\begin{align*}
\Psi & := \frac{1}{2} \| \nabla \hat \Lam \|^2 - (S,\hat \Lam) - B( \nabla \hat \Lam,  (m_1 - m_2)
 ( \nabla J \star \ph_1) + m_2  (\nabla J \star \ph)) \\
& \geq \frac{1}{4} \| \nabla \hat \Lam \|^2 - C ( \| \ph \|^2 + \| \s \|^2 ).
\end{align*}
Substituting \eqref{Cts:4}-\eqref{Cts:6} into \eqref{Cts:7} yields
\begin{equation}\label{Cts:7a}
\begin{aligned}
\frac{d}{dt} \Psi + \frac{A \alpha_0}{2} \| \dt \ph \|^2 & \leq C \big ( 1+ \| \dt \ph_1 \|_4^4 + \| \dt \ph_2 \|_4^4 \big ) \big ( \| \ph \|^2 + \| \nabla \hat \Lam \|^2 \big ) \\
& \quad + C \| \s \|_V^2 + \eps_2 \| \dt \s \|^2 + C_{\eps_2} \| \ph \|_V^2,
\end{aligned}
\end{equation}
after choosing $\delta$ sufficiently small. Integrating \eqref{Cts:7a}
over $(0,t)$, for arbitrary $t\in(0,T)$, and noting
\begin{align*}
\Psi(0) & \leq C \| \hat \Lam(0) \|_V^2 + C \| \ph_{0,1} - \ph_{0,2} \|^2 + C \| \s_{0,1} - \s_{0,2} \|^2 \\
& \leq C \| \ph_{0,1} - \ph_{0,2} \|_V^2 + C\| \s_{0,1} - \s_{0,2} \|^2 =: \mathcal{Y}^2,
\end{align*}
together with \eqref{reg_superstrong} and \eqref{cts1} we find that
\begin{align*}
& \| \nabla \hat \Lam(t) \|^2 + \int_0^t \| \dt \ph \|^2 \, ds  \\
& \quad \leq  C \int_0^t   \big ( 1+ \| \dt \ph_1 \|_4^4 + \| \dt \ph_2 \|_4^4 \big ) \big ( \| \ph \|^2 + \| \nabla \hat \Lam \|^2 \big )  + \| \s \|_V^2 + \eps_2 \| \dt \s \|^2
+C_{\eps_2} \| \varphi \|_V^2 \, ds \\
& \qquad  + C ( \| \s(t) \|^2 + \| \ph(t) \|^2) + \Psi(0) \\
& \quad \leq  C \int_0^t ( 1 + \| \dt \ph_1 \|_4^4 + \| \dt \ph \|_4^4 ) \|\nabla \hat \Lam \|^2 + \eps_2 \| \dt \s \|^2 ds +\mathcal{Y}^2
\end{align*}
on account of
\begin{align*}
\| \ph \|_{L^\infty(0,T;H)}
+\|\ph \|_{L^2(0,T;V)}+\|\s \|_{L^\infty(0,T;H)} + \| \s \|_{L^2(0,T;V)} \leq \mathcal{Y},
\end{align*}
and of the fact that $\dt \ph_1, \dt \ph_2 \in L^4(0,T;L^4(\Omega))$. By Gronwall's lemma in integral form, for $t \in (0,T)$, we get
\begin{equation}\label{Cts:8}
\begin{aligned}
\| \nabla \hat \Lam(t) \|^2 + \int_0^t \| \dt \ph \|^2 \, ds & \leq \mathcal{Y}^2 + C \eps_2 \int_0^t \| \dt \s \|^2 \, ds.
\end{aligned}
\end{equation}
With this, we now derive estimates for $\hat \Lam$ in $L^2(0,T;W)$.  Observe that by elliptic regularity and \eqref{cts:Lam:2}, we have
\begin{equation}\label{LamH2}
\begin{aligned}
\| \hat \Lam \|_{W} & \leq C_2 \big ( \| \Delta \hat \Lam \| + \| \hat \Lam \|_V + \| \nabla \hat \Lam \cdot \bold{n} \|_{H^{1/2}(\Gamma)} \big ) \\
& \leq C_2 \big ( \| \Delta \hat \Lam \| + \| \nabla \hat \Lam \| + \| \ph \| + \| \nabla \hat \Lam \cdot \bold{n} \|_{H^{1/2}(\Gamma)} \big )
\end{aligned}
\end{equation}
for a constant $C_2 > 0$ depending only on $\Omega$.  Taking the difference of \eqref{eq_first},
using Lemma \ref{controlofdiv} and testing with $\Delta \hat \Lam$, we obtain
\begin{align*}
\| \Delta \hat \Lam \|^2 & = ( \dt \ph - S, \Delta \hat \Lam) +B  (m'_2 \nabla \ph_2
\cdot( \nabla J \star \ph) + m_2 \div (\nabla J \star \ph), \Delta \hat \Lam) \\
& \quad + B ( (m_1 - m_2) \div(\nabla  J \star \ph_1)+ m'_2 \nabla \ph
\cdot( \nabla J \star \ph_1) , \Delta \hat \Lam) \\
& \quad + B( (m'_1 - m'_2) \nabla \ph_1 \cdot(\nabla J \star \ph_1), \Delta \hat \Lam) \\
& \leq \tfrac{1}{2} \| \Delta \hat \Lam \|^2 +  C \big (\| \dt \ph \|^2 + \| \ph \|_V^2 + \| \s \|^2 \big ),
\end{align*}
which, owing to \eqref{ass:cts:s:m:n} implies that
\begin{align*}
\| \Delta \hat \Lam \| \leq C \big (\| \dt \ph \| + \| \ph \|_V + \| \s \|\big ).
\end{align*}
Next, we recall \eqref{bc}, so that
\begin{align*}
\nabla \hat \Lam \cdot \bold{n} = B[(m_1 - m_2) \nabla J \star \ph_1 + m_2 \nabla J \star \ph ] \cdot \bold{n},
\end{align*}
and by invoking Lemma~\ref{LEM_trace},
we can estimate the $H^{1/2}(\Gamma)$- norm of this boundary term in the same fashion as in
\cite[Proof of Lemma 4, (3.33)--(3.37)]{FGS}. More precisely, we have that
\begin{align*}
& \|\nabla \hat \Lam \cdot \bold{n} \|_{H^{1/2}(\Gamma)} \\
& \quad \leq B \| \hat m \|_{L^{\infty}(\Gamma)} \|  (\nabla J \star \ph_1) \cdot \bold{n} \|_{H^{1/2}(\Gamma)} + B \| \hat m \|_{H^{1/2}(\Gamma)} \| (\nabla J \star \ph_1) \cdot \bold{n} \|_{L^{\infty}(\Gamma)} \\
& \qquad + B \| m_2 \|_{L^{\infty}(\Gamma)} \|(\nabla J \star \ph) \cdot \bold{n} \|_{H^{1/2}(\Gamma)} + B \| m_2 \|_{H^{1/2}(\Gamma)} \| (\nabla J \star \ph) \cdot \bold{n} \|_{L^{\infty}(\Gamma)},
\end{align*}
where $\hat m := m_1 - m_2$.  Invoking Lemma~\ref{LEM_controlofdiv}, the boundedness of the normal vector $\bold{n}$ and the trace theorem, it holds that
\begin{align*}
 \| (\nabla J \star \ph_1) \cdot \bold{n} \|_{L^{\infty}(\Gamma)} &\leq \| \nabla J \star \ph_1 \|_{L^\infty(\Gamma)} \leq \| \nabla J \star \ph_1 \|_{W^{1,3}(\Omega)} \leq C \|\ph_1 \|_{L^3(\Omega)} \leq C, \\
 \| (\nabla J \star \ph) \cdot \bold{n} \|_{L^\infty(\Gamma)} &\leq C \| \ph \|_V, \\
 \| (\nabla J \star \ph_1) \cdot \bold{n} \|_{H^{1/2}(\Gamma)} &\leq C \| J \star \ph_1 \|_{W} \leq C, \\
 \| (\nabla J \star \ph) \cdot \bold{n} \|_{H^{1/2}(\Gamma)}  &\leq C \| J \star \ph \|_{W} \leq C \|\ph \|.
\end{align*}
On the other hand, thanks to Agmon's inequality \eqref{agmon}, we have that
\begin{align*}
\| \hat m \|_{H^{1/2}(\Gamma)} & \leq C \| \hat m \|_V \leq C \| \ph \|_V, \\
\| \hat m \|_{L^{\infty}(\Gamma)} &\leq C \| \ph \|_{L^{\infty}(\Gamma)} \leq C \| \ph \|_{L^{\infty}(\Omega)}
\leq C\Vert\ph\Vert^{1/2}\Vert\ph\Vert_W^{1/2}.
\end{align*}
Hence, we find that, for $\delta > 0$ to be determined later,
\begin{align*}
\| \nabla \hat \Lam \cdot \bold{n} \|_{H^{1/2}(\Gamma)} \leq \frac{\delta}{C_2} \| \ph \|_{W} + C_\delta \| \ph \|_V,
\end{align*}
and thus from the elliptic estimate \eqref{LamH2}, we have for any $\delta > 0$,
\begin{align}\label{Cts:Lam:H2}
\| \hat \Lam \|_{W} \leq C_\delta \big (\| \ph \|_V + \| \nabla \hat \Lam \| + \| \s \| + \| \dt \ph \| \big )+ \delta \| \ph \|_{W}.
\end{align}
Next, we employ the identity
$A \partial_i \ph_k = \lam^{-1}(\ph_k) \partial_i \Lam(\ph_k)$ for $k =1,2$
and $i=1,2$, to deduce that
\begin{align*}
\notag A \partial_{i j}^2 \ph& = \frac 1 {\lam(\ph_1)}\partial_{i j}^2 \hat \Lam	+  \frac {\lam(\ph_2)-\lam(\ph_1)} {\lam(\ph_1)\lam(\ph_2)} \partial_{i j}^2 \Lam(\ph_2)\\
 & \notag \quad	- \frac {\lam^2(\ph_2)-\lam^2(\ph_1)} {\lam^2(\ph_1)\lam^2(\ph_2)} \partial_{i} \lam(\ph_1)\partial_{j}\Lam(\ph_2)- \frac 1 {\lam^2(\ph_2)}(\partial_{i}\lam(\ph_1)-\partial_{i}\lam(\ph_2))\partial_{j}\Lam(\ph_1)	\\
  & \notag \quad - \frac 1 {\lam^2(\ph_2)}\partial_{i}\lam(\ph_1)\partial_{j}\hat \Lam
	\quad \text{ for } i,j=1,2.
\end{align*}
Employing Agmon's inequality \eqref{agmon}, the Gagliardo--Nirenberg inequality, as well as \eqref{reg_superstrong}, \eqref{Cts:Lam:H2}, the $L^\infty(0,T;W)$-regularity of $\Lam(\ph_1)$ and $\Lam(\ph_2)$ from \eqref{inftyH2},
 and the assumption $\lambda\in C^{1,1}([-1,1])$, we infer that
\begin{align*}
\| \ph \|_{W} & \leq C \| \hat \Lam \|_{W} + C \| \ph \|_{L^\infty(\Omega)} \| \Lam(\ph_2) \|_{W}  + C \| \ph \|_{6} \| \nabla \ph_1 \|_{6} \| \nabla \Lam(\ph_2) \|_{6} \\
& \quad + C \| \nabla \ph \|_{4} \| \nabla \Lam(\ph_2) \|_{4} +
C \|\ph \|_6 \| \nabla \ph_1 \|_6 \| \nabla \Lam(\ph_2) \|_6
+ C \| \nabla \ph_1 \|_{4} \| \nabla \hat \Lam \|_{4} \\
& \leq C \| \ph \|_V + C \| \ph \|_V^{1/2} \| \ph \|_{W}^{1/2} + C \| \hat \Lam \|_{W} \\
& \leq C_\delta \big ( \| \ph \|_V + \| \nabla \hat \Lam \| + \| \s \| + \| \dt \ph \| \big ) + C \delta  \| \ph \|_{W}.
\end{align*}
Choosing $\delta$ sufficiently small, we then obtain
\begin{align}\label{phH2}
\| \ph \|_{W} \leq C \big ( \| \ph \|_V + \| \nabla \hat \Lam \| + \| \s \| + \| \dt \ph \| \big ).
\end{align}
Moreover, taking the difference of \eqref{eq_second} and testing with $\Delta \s$ yields
\begin{align*}
(n_2\Delta \s, \Delta \s) & = ( \dt \s -n'_2 \nabla \ph_2 \cdot \nabla \s
- (n'_1 - n'_2) \nabla \ph_1\cdot \nabla \s_1, \Delta \s) \\
& \quad + ( S - n'_2 \nabla \ph \cdot \nabla \s_1 - (n_1 - n_2) \Delta \s_1 , \Delta \s)  \\
& \leq \frac{n_*}{2} \| \Delta \s \|^2 + C \big ( \| \dt \s \|^2 + \| \nabla \ph_2 \|_4^2 \| \nabla \s \|_4^2 + \| \nabla \ph_1 \|_6^2 \| \nabla \s_1 \|_6^2 \| \ph \|_6^2 \big ) \\
& \quad  + C \big (\| \nabla \s_1 \|_4^2 \| \nabla \ph \|_4^2 + \| \Delta \s_1 \|^2 \| \ph \|_{L^{\infty}(\Omega)}^2 + \| \ph \|^2 + \| \s \|^2 \big ) \\
& \leq \frac{n_*}{2} \| \Delta \s \|^2 + C_\delta \big ( \| \dt \s \|^2 + \| \s \|_V^2 + \| \ph \|_{W}^2 \big )  + \delta \| \s \|_{W}^2,
\end{align*}
so that elliptic regularity entails that
\begin{align*}
\| \s \|_{W}^2  &\leq C \big ( \| \Delta \s \|^2 + \| \s \|_V^2 \big ) \\
& \leq C_\delta  \big ( \| \dt \s \|^2 + \| \s \|_V^2 + \| \ph \|_{W}^2 \big )  +C \delta \| \s \|_{W}^2,
\end{align*}
and with $\delta$ sufficiently small we obtain
\begin{align}\label{sH2}
\| \s \|_{W}^2 \leq C \big ( \| \dt \s \|^2 + \| \s \|_V^2 + \| \ph \|_{W}^2  \big ) .
\end{align}
First, we apply Gronwall's inequality to \eqref{Hcts:1} and then substitute \eqref{sH2}, which leads to
\begin{align*}
\| \nabla \s (t) \|^2 + \int_0^t \| \dt \s \|^2 \, ds &  \leq C \| \s_{0,1} - \s_{0,2} \|_V^2 + C \int_0^t \| \ph \|_{W}^2 + \| \s \|^2 + \eps_1 \| \s \|_{W}^2 \, ds \\
& \leq C \| \s_{0,1} - \s_{0,2} \|_V^2 + C \int_0^t \| \ph \|_{W}^2 + \| \s \|_V^2 + \eps_1 \| \dt \s \|^2 \, ds .
\end{align*}
Choosing $\eps_1$ sufficiently small, and then substituting \eqref{phH2} yields
\begin{align*}
\| \nabla \s(t) \|^2  + \int_0^t \| \dt \s \|^2  \, ds & \leq C \| \s_{0,1} - \s_{0,2} \|_V^2 + \mathcal{Y}^2 + C_3 \int_0^t \| \nabla \hat \Lam \|^2 + \| \dt \ph \|^2 \, ds \\
& \leq C \| \s_{0,1} - \s_{0,2} \|_V^2 + \mathcal{Y}^2 + C \| \ph \|_{L^2(0,T;V)}^2 + C_3 \int_0^t \| \dt \ph \|^2 \, ds \\
& \leq \underbrace{C \| \s_{0,1} - \s_{0,2} \|_V^2 + \mathcal{Y}^2}_{=: \mathcal{Z}^2} + C_3 \int_0^t \| \dt \ph \|^2 \, ds,
\end{align*}
for some positive constant $C_3$, \last{on account of  \eqref{cts:Lam} and \eqref{cts1} for the integral term involving $\| \nabla \hat \Lam \|^2$}.  To the above, we add the inequality \eqref{Cts:8} multiplied by $(C_3 + 1)$, leading to
\begin{align*}
\| \nabla \s(t) \|^2 + \| \nabla \hat \Lam(t) \|^2 + \int_0^t \| \dt \s \|^2 + \| \dt \ph \|^2  \, ds \leq \mathcal{Z}^2 + C(C_3 +1) \eps_2 \int_0^t \| \dt \s \|^2 \, ds.
\end{align*}
Choosing $\eps_2$ sufficiently small, we then obtain
\begin{align*}
\| \nabla \s \|_{L^\infty(0,T;H)}^2 + \| \nabla \hat \Lam \|_{L^\infty(0,T;H)}^2 + \| \dt \s \|_{L^2(0,T;H)}^2 + \| \dt \ph \|_{L^2(0,T;H)}^2 \leq \mathcal{Z}^2.
\end{align*}
Returning to \eqref{phH2} and \eqref{sH2} this implies
\begin{align*}
\| \ph \|_{L^2(0,T;W)}^2 + \| \s \|_{L^2(0,T;W)}^2 \leq C \big ( \| \ph_{0,1} - \ph_{0,2} \|_{V}^2 + \| \s_{0,1} - \s_{0,2} \|_{V}^2 \big ),
\end{align*}
which completes the proof. \qed

\section{Application to the inverse identification problem}\label{sec:appl}
In this section we study the constrained minimisation problem \eqref{opt}.  The standard procedure to deriving first-order optimality conditions is to first establish the differentiability of the solution operator $S: \ph_0 \mapsto \ph$, \last{show} well-posedness results for the corresponding linearised system and adjoint system, and then use them to derive the optimality condition.

We restrict our analysis to the two dimensional case.  In preliminary calculations not shown here, it appears that the Fr\'echet differentiability of the solution operator would require a continuous dependence estimate for $\ph$ and $\s$ in $L^4(0,T;V) \cap L^2(0,T;W)$.  While this is guaranteed for strong solutions given by Theorem \ref{thm:superstrong}, there is a requirement on the initial condition $\ph_0 \in W$ to fulfil the compatibility condition (c.f.~Theorem \ref{thm:superstrong})
\begin{align}\label{ph0:com}
[\nabla \Lam(\ph_0) - B m(\ph_0) (\nabla J \star \ph_0)] \cdot \bold{n} = 0 \text{ on } \Gamma.
\end{align}
Hence, as a first attempt we can take the space of admissible controls as
\begin{align}
&U = \Big \{ u  \in W\, : \, |u| \leq 1 \text{ a.e.~in } \Omega \text{ and } \eqref{ph0:com} \mbox{ holds }\Big \}.
\label{defU}
\end{align}
However, this set is not convex due to the nonlinear constraint \eqref{ph0:com}, and the resulting optimality condition would involve Lagrange multipliers, which further complicates the numerical implementation.  Moreover, instead of using the norm $\|\cdot \|_{H^1(\Omega)}$ in \eqref{opt},
one would employ a norm $\|\cdot \|_U$ which need to match with the expected regularity of the control.  In this case we can choose for example
\begin{align*}
\| u \|_U := \big (\| u \|^2 + \| \nabla u \|^2 +  \| \Delta u \|^2 \big )^{1/2},
\end{align*}
but we can expect \last{that} the strong form of the optimality condition involves a fourth order differential operator.  Hence, in light of both analytical and numerical complications arising from working with solutions of the highest level of regularity, we consider regularities obtained from Theorem~\ref{thm:strong}.  It turns out that they are enough to establish G\^ateaux differentiability of the solution operator, and we can consider, for $\kappa>0$ fixed, the convex set
\begin{align}\label{consp}
U := \{ u \in V \, : \, |u| \leq 1\,,\,\,\, M(u)\leq \kappa \,\,\, \text{ a.e.~in } \Omega \},
\end{align}
with the entropy function $M$ defined as in \eqref{ass:ini},
as the set of admissible controls. It is immediate to check that $U$ is closed in $H$. Notice also that,
if $M$ is bounded in $[-1,1]$ (this occurs, in particular, if $m$ is weakly degenerate, i.e.,
if $m(\pm 1)=0$ with order strictly less than $2$, see \cite[Remark 8]{FGR2}), the condition
$M(u)\leq\kappa$ in the definition of the set $U$ can be removed. From now onward, we will tacitly assume $U$ to be defined by \eqref{consp}. Moreover, assuming that \eqref{ass:m:F}--\eqref{ass:P} and \eqref{ass:str:m:n}--\eqref{ass:str:P} hold, and that $T>0$ and $\s_0\in V$ are fixed,
the control-to-state operator $S:U\subset V\to V$ is well defined by $S(\varphi_0)=\varphi(T)$,
for all $\varphi_0\in U$, where $[\varphi,\sigma]$ is the unique solution to (P) corresponding to $[\varphi_0,\s_0]$
and given by Theorem \ref{thm:strong} (uniqueness is ensured by Corollary \ref{cor:uniq}).

\subsection{Analysis of the linearised system and adjoint system}
Fix $\bar \ph_0 \in U$ (fix also $\s_0\in V$) and let $[\bar \ph, \bar \s]$ denote the unique solution to $(\mathrm{P})$ obtained from Theorem \ref{thm:strong} corresponding to initial conditions
$[\bar \ph_0, \s_0]$.  For $h \in L^\infty(\Omega) \cap (U - U)$, i.e., $h =v - \bar \ph_0$ for some $v \in U$, we consider the linearised system written in strong form
\begin{subequations}\label{linearised}
\begin{alignat}{3}
& \notag \dt \xi - \div (A \lam(\bar \ph) \nabla \xi + A \lam'(\bar \ph) \xi \nabla \bar \ph)&& \\
&\notag \quad \qquad + B \, \div (m(\bar \ph) (\nabla J \star \xi) + m'(\bar \ph)\xi (\nabla J \star \bar \ph)) && \\
& \qquad = - A(PF')'(\bar \ph) \xi + P'(\bar \ph)\xi (\bar \s + B J \star \bar \ph) + P(\bar \ph)(\eta + B J \star \xi) && \text{ in } Q, \\
& \notag \dt \eta - \div (n(\bar \ph) \nabla \eta + n'(\bar \ph) \xi \nabla \bar \s) && \\
& \qquad =  A(PF')'(\bar \ph) \xi - P'(\bar \ph)\xi (\bar \s + B J \star \bar \ph) - P(\bar \ph)(\eta + B J \star \xi) && \text{ in } Q, \\
& [A \lam (\bar \ph) \nabla \xi + A \lam'(\bar \ph) \xi \nabla \bar \ph - B m(\bar \ph) (\nabla J \star \xi) - Bm'(\bar \ph) \xi (\nabla J \star \bar \ph)] \cdot \bold{n} = 0 && \text{ on } \Sigma, \\
& [n(\bar \ph) \nabla \eta + n'(\bar \ph) \xi \nabla \bar \s] \cdot \bold{n} = 0 && \text{ on } \Sigma, \\
& \xi(0) = h, \quad \eta(0) = 0 && \text{ in } \Omega,
\end{alignat}
\end{subequations}
as well as the adjoint system written in strong form
\begin{subequations}\label{adjoint}
\begin{alignat}{3}
\notag & - \dt p - \div (A \lam(\bar \ph) \nabla p) + A \lam'(\bar \ph) \nabla \bar \ph \cdot \nabla p - B m'(\bar \ph) (\nabla J \star \bar \ph) \cdot \nabla p && \\
\notag & \quad \qquad - B \nabla J \dot{\star} (m(\bar \ph) \nabla p) + P'(\bar \ph)(\bar \s - A F'(\bar \ph) +  B J \star \bar \ph)(q-p) && \\
\notag & \quad \qquad - A P(\bar \ph) F''(\bar \ph) (q-p) + B J \star (P (\bar \ph) (q-p)) + n'(\bar \ph) \nabla \bar \s \cdot \nabla q && \\
&\quad  \quad = 0 && \text{ in } Q, \\
& - \dt q - \div (n(\bar \ph) \nabla q) + P(\bar \ph)(q-p)  = 0 && \text{ in } Q, \\
& A \lam(\bar \ph)\dn p = n(\bar \ph) \dn q = 0 && \text{ on } \Sigma, \\
& p(T) = \bar \ph(T) - \ph_\Omega, \quad q(T) = 0 && \text{ in } \Omega,
\end{alignat}
\end{subequations}
where $\nabla J \dot{\star} \nabla p$ is defined as
\begin{align*}
	(\nabla J \dot{\star} \nabla p)(x) := \iO \nabla J(x-y )\cdot \nabla p(y) dy
	\quad \hbox{for $a.e.$ $x \in \Omega.$}
\end{align*}
We note that both the linearised system \eqref{linearised} and the adjoint system \eqref{adjoint} are linear in their respective variables $[\xi, \eta]$ and $[p, q]$.  The weak well-posedness of both systems are formulated as follows.

\begin{thm}\label{thm:lin:adj}
Assume \eqref{ass:m:F}-\eqref{ass:ini}, \eqref{ass:str:m:n}-\eqref{ass:str:P} are satisfied with $\bar \ph_0 \in V$, and with $\s_0\in V$ fixed. Denote by $[\bar \ph, \bar \s]$ the unique strong solution to $(\mathrm{P})$ corresponding to $[\bar \ph_0,\s_0]$ and obtained from Theorem \ref{thm:strong}.  For any $v \in U $, setting $h = v - \bar \ph_0$, then there exists a unique solution $[\xi, \eta]$ to the linearised system \eqref{linearised} corresponding to $(h, \bar \ph, \bar \s)$, and for any $\ph_\Omega \in H$ a unique solution $[p, q]$ to the adjoint system \eqref{adjoint} corresponding to $(\ph_\Omega, \bar \ph, \bar \s)$ in the following sense:
\begin{itemize}
\item they possess the following regularities
\begin{align}\label{reg:lin:adj}
\xi, \eta, q & \in H^1(0,T;V^*) \cap L^2(0,T;V)  \cap C^0([0,T];H), \\
p & \in W^{1,\frac{4}{3}}(0,T;V^*) \cap L^2(0,T;V) \cap L^\infty(0,T;H) \cap C^0([0,T];V^*); \label{reg:p}
\end{align}
\item for every $w, z \in V$ and almost every $t \in (0,T)$ we have
\begin{subequations}\label{lin:weak}
\begin{alignat}{2}
\notag & \< \dt \xi, w>_V + \iO \big ( A\lam(\bar \ph) \nabla \xi + A \lam'(\bar \ph) \xi \nabla \bar \ph \big ) \cdot \nabla w \\
\notag & \quad - \iO  B\big ( m(\bar \ph) (\nabla J \star \xi) + m'(\bar \ph) \xi (\nabla J \star \bar \ph) \big ) \cdot \nabla w
+ \iO A(PF')'(\bar \ph) \xi w \\
& \quad - \iO (P'(\bar \ph)(\bar \s + B J \star \bar \ph) \xi + P(\bar \ph)(\eta + B J \star \xi)) w = 0, \label{lin:weak:1} \\
\notag & \< \dt \eta, z>_V + \iO \big (n(\bar \ph) \nabla \eta + n'(\bar \ph) \xi \nabla \bar \s \big ) \cdot \nabla z \\
& \quad + \iO \big (P'(\bar \ph)(\bar \s + B J \star \bar \ph) \xi + P(\bar \ph)(\eta + B J \star \xi) - A (PF')'(\bar \ph) \xi \big ) z = 0, \label{lin:weak:2}
\end{alignat}
\end{subequations}
and
\begin{subequations}\label{adj:weak}
\begin{alignat}{2}
\notag & - \< \dt p, w>_V + \iO A \lam(\bar \ph) \nabla p \cdot \nabla w
 +\iO \big ( A \lambda'(\bar \ph) \nabla \bar \ph \cdot \nabla p - Bm'(\bar \ph)( \nabla J \star \bar \ph) \cdot \nabla p \big ) w \\
\notag & \quad - \iO \big ( B \nabla J \dot{\star} (m(\bar \ph) \nabla p) + A (PF'')(\bar \ph)(q-p)  - n'(\bar \ph)(\nabla \bar \s \cdot \nabla q) \big ) w \\
& \quad + \iO \big (P'(\bar \ph)(\bar \s - A F'(\bar \ph) + B J \star \bar \ph)(q-p)  + B J \star (P(\bar \ph)(q-p)) \big )w = 0, \label{adj:weak:1} \\
& -\< \dt q, z>_V + \iO n(\bar \ph) \nabla q \cdot \nabla z +\iO P(\bar \ph)(q-p) z = 0, \label{adj:weak:2}
\end{alignat}
\end{subequations}
along with the initial/terminal conditions
\begin{align*}
\xi(0) = h, \quad \eta(0) = 0, \quad q(T) = 0 \text{ in } H, \\
\<p(T), w >_V = \< \bar \ph(T) - \ph_\Omega, w >_V \quad\last{\forall w \in V}.
\end{align*}
\end{itemize}
\end{thm}

\begin{proof}
We proceed formally by deriving sufficient estimates for a Faedo--Galerkin approximation.  Focusing first on the linearised system, testing \eqref{lin:weak:1} with $\xi$ and \eqref{lin:weak:2} with $\eta$, integrating over $\Omega$ and summing the equations leads to
\begin{align}\label{lin:est}
\notag & \frac{1}{2} \frac{d}{dt}  \big ( \| \xi \|^2 + \| \eta \|^2 \big ) + \iO A \lam(\bar \ph) |\nabla \xi|^2
+ \iO n( \bar \ph) |\nabla \eta|^2 \\
& \quad \leq C
\iO \big(|\xi| |\nabla \bar \ph||\nabla \xi | + |\nabla J \star \xi| |\nabla \xi| +  |\xi||\nabla J \star \bar \ph| |\nabla \xi | + |\xi | |\nabla \bar \s | | \nabla \eta |\big) \\
\notag & \qquad  + C \iO (| \bar \s| + |J \star \bar \ph|)(|\xi|^2 + |\xi ||\eta|) + C \iO \big(|\eta | |\xi | + |J \star \xi | (|\xi|+ |\eta|) + |\xi|^2 + |\eta|^2\big),
\end{align}
where we have used the boundedness of $\lam', n', m, m', P, P'$ and $(PF')'$.
Employing the Gagliardo--Nirenberg inequality and the regularities $\bar \ph, \bar \s \in L^\infty(0,T;V) \cap L^2(0,T;W)$ from Theorem \ref{thm:strong}, we see that
\begin{align*}
& \frac{1}{2} \frac{d}{dt}  \big ( \| \xi \|^2 + \| \eta \|^2 \big ) + A \alpha_0 \|\nabla \xi\|^2 + n_* \|\nabla \eta\|^2 \\
& \quad \leq  \frac{A \alpha_0}{2} \| \nabla \xi\|^2 + \frac{n_*}{2} \| \nabla \eta \|^2 + C \big ( 1 + \| \nabla \bar \ph \|_{4}^4 + \| \nabla \bar \s \|_4^4 + \| \bar \s \|_{\infty}^2 \big )\big ( \| \xi \|^2 + \| \eta \|^2 \big ).
\end{align*}
Invoking Gronwall's inequality yields the $L^\infty(0,T;H) \cap L^2(0,T;V)$-estimate for $[\xi, \eta]$.  Moreover, from \eqref{lin:weak:1} and \eqref{lin:weak:2} it is easy to see that
\begin{align*}
\| \dt \xi \|_{V^*} & \leq C \big (\| \xi \|_V +  \| \nabla \bar \ph \|_4 \| \xi \|_4 + \| \eta \| \big ), \\
\|\dt \eta \|_{V^*} & \leq C \big (  \| \eta \|_V +  \| \nabla \bar \s \|_4  \| \xi \|_4 + \| \xi \| \big ),
\end{align*}
which leads to the $H^1(0,T;V^*)$-estimate for $[\xi, \eta]$.  For uniqueness, we just mention that the testing procedure leading to \eqref{lin:est} is valid with the regularity stated in \eqref{reg:lin:adj}, and thanks to the linearity of the linearised system \eqref{linearised}, the difference of two solutions satisfy \eqref{lin:weak} with zero initial conditions.  Then, by a similar argument with Gronwall's inequality we can infer the uniqueness of weak solutions.

For the adjoint system, again we derive formal estimates and mention that the argument can be made rigorous with a Faedo--Galerkin approximation.  Testing \eqref{adj:weak:1} with $p$ and \eqref{adj:weak:2} with $q$, integrating over $(t, T)$ for arbitrary $t\in (0,T)$,
using the the Gagliardo--Nirenberg inequality and the terminal conditions and after summing the equations we arrive at
\begin{align*}
& \frac{1}{2} \big ( \| p(t) \|^2 + \| q(t) \|^2 \big ) + \int_{t}^T A \alpha_0 \| \nabla p \|^2 + n_* \| \nabla q \|^2 + \| P^{1/2}(\bar \ph) q \|^2 \, ds\\
& \quad \leq  \frac{1}{2} \|  \bar \ph(T) - \ph_\Omega \|^2  + C \int_t^T \iO \big (|\nabla \bar \ph | + |\nabla J \star \bar \ph| \big ) |\nabla p | |p|  + |\nabla J \dot{\star} (m(\bar \ph) \nabla p)| |p|\, ds \\
& \qquad  + C \int_t^T \iO  |\nabla \bar \s | |\nabla q | |p| + \big ( 1 + |\bar \s |) ( |q| + |p|) |p|  \, ds\\
& \quad \leq \int_t^T  \frac{A \alpha_0}{2} \| \nabla p\|^2 + \frac{n_*}{2} \| \nabla q \|^2
+ C \big ( 1+ \| \nabla \bar \ph \|_4^4  + \| \nabla \bar \s \|_4^4 + \| \bar \s \|_\infty^2 \big ) \big ( \|p \|^2 + \| q \|^2 \big ) \, ds \\
& \qquad + \frac{1}{2} \| \bar \ph(T) - \ph_\Omega \|^2.
\end{align*}
Applying the integral form of Gronwall's inequality and the regularities of $[\bar \ph, \bar \s]$ from Theorem \ref{thm:strong} yields the existence of a constant $C$ independent of $p$ and $q$ such that
\begin{align}\label{p:q:est}
\sup_{t \in (0,T)} \big ( \| p(t) \|^2 + \| q(t) \|^2 \big ) + \int_0^T \| \nabla p \|^2 + \| \nabla q \|^2 ds \leq C\| \bar \ph(T) - \ph_\Omega \|^2,
\end{align}
leading to the $L^\infty(0,T;H) \cap L^2(0,T;V)$-estimate for $[p, q]$.  In a similar fashion, we deduce from \eqref{adj:weak:1} and \eqref{adj:weak:2} that
\begin{align*}
\| \dt p \|_{V^*} &\leq C \big (1 + \| \bar \ph \|_W^{1/2} \big ) \| \nabla p \| + C \| \bar \s \|_W^{1/2} \| \nabla q \| + C \big (\| q \|_V + \| p \|_V \big ), \\
\| \dt q \|_{V^*} & \leq C \big ( \| q \|_V + \|p \| \big )
\end{align*}
leading to the $H^1(0,T;V^*)$-estimate for $q$ and the $W^{1,\frac{4}{3}}(0,T;V^*)$-estimate for $p$, \last{the latter is due to the fact that $\bar \ph \in L^2(0,T;W)$ from the statement of Theorem \ref{thm:strong}}.  For uniqueness of solutions we use the fact that any weak solution to the adjoint system satisfies the inequality \eqref{p:q:est} due to weak lower semicontinuity of the Bochner norms.  If $\hat p = p_1 - p_2$ and $\hat q = q_1 - q_2$ are the difference of two weak solutions corresponding to the same terminal data $\bar \ph(T) - \ph_\Omega$, thanks to the linearity of \eqref{adjoint}, we see that $[\hat p, \hat q]$ satisfies \eqref{adj:weak:1}-\eqref{adj:weak:2} with $\hat p(T) = \hat q(T) = 0$.  Then, the analogous inequality \eqref{p:q:est} for $[\hat p, \hat q]$ yields that $\hat p = \hat q = 0$ for a.e.~$(x,t) \in Q$.
\end{proof}

An immediate consequence of Theorem \ref{thm:lin:adj} is the following.

\begin{cor}
Assume that \eqref{ass:m:F}--\eqref{ass:ini}, \eqref{ass:str:m:n}--\eqref{ass:str:P} are satisfied. Let $\ph_\Omega \in H$, and $[\bar \ph, \bar \s]$ be the unique strong solution to $(\mathrm{P})$ corresponding to initial data $[\bar \ph_0, \s_0] \in U \times V$ obtained from Theorem \ref{thm:strong}.  For any $v \in U$ with $h = v - \bar \ph_0$, let $[\xi, \eta]$ and $[p, q]$ denote the unique weak solutions to the linearised system \eqref{linearised} and adjoint system \eqref{adjoint}, respectively.  Then, it holds that
\begin{align}\label{cor:opt}
 \iO (\bar \ph(T) - \ph_\Omega) \xi(T)  = \iO p(0) h.
\end{align}
\end{cor}

\begin{proof}
Testing \eqref{lin:weak:1} with $w = p$ and \eqref{lin:weak:2} with $z = q$ and integrating over time
leads to
\begin{subequations}
\begin{alignat}{2}
\notag &\int_0^T \< \dt \xi, p>_V + \int_Q \big ( A \lam(\bar \ph) \nabla \xi + A \lam'(\bar \ph) \xi \nabla \bar \ph \big ) \cdot \nabla p \\
\notag & \quad - \int_Q B \big ( m(\bar \ph)(\nabla J \star \xi) + m'(\bar \ph) \xi (\nabla J \star \bar \ph) \big ) \cdot \nabla p
+\int_Q  A(PF')'(\bar \ph) \xi p\\
\label{adj:lin:1} & \quad - \int_Q \big ( P'(\bar \ph)(\bar \s + B J \star \bar \ph) \xi + P(\bar \ph) (\eta + B J \star \xi) \big ) p = 0, \\
\notag &\int_0^T \< \dt \eta, q >_V + \int_Q  \big ( n(\bar \ph) \nabla \eta + n'(\bar \ph) \xi \nabla \bar \s \big ) \cdot \nabla q \\
\label{adj:lin:2} & \quad + \int_Q \big ( P'(\bar \ph)(\bar \s + B J \star \bar \ph) \xi + P(\bar \ph) (\eta + B J \star \xi) - A(PF')'(\bar \ph) \xi \big ) q = 0.
\end{alignat}
\end{subequations}
On the other hand, taking $w=\xi$ in \eqref{adj:weak:1}, $z=\eta$ in \eqref{adj:weak:2}, and integrating the resulting identities by parts
on $(0,T)$, we get
\begin{subequations}
\begin{alignat}{2}
\notag & \int_0^T \< \dt \xi, p>_V + \int_Q \big (A \lam(\bar \ph) \nabla \xi  + A \lam'(\bar \ph) \xi \nabla \bar \ph \big )\cdot \nabla p \\
\notag & \quad - \int_Q B \big (m(\bar \ph) (\nabla J \star \xi)  + m'(\bar \ph)\xi (\nabla J \star \bar \ph) \big ) \cdot \nabla p
+ \int_Q A(PF')'(\bar \ph) \xi (p  - q) \\
\notag & \quad - \int_Q  P'(\bar \ph)(\bar \s + B J \star \bar \ph)(p-q) \xi
-\int_Q P(\bar \ph) (BJ \star \xi)(p-q) \\
\label{adj:lin:3} & \quad + \int_Q n'(\bar \ph)\xi (\nabla \bar \s \cdot \nabla q )= \< \bar \ph(T) - \ph_\Omega, \xi(T)>_V - \iO p(0) \xi(0), \\
\label{adj:lin:4} & \int_0^T \< \dt \eta, q>_V + \int_Q n(\bar \ph) \nabla q \cdot \nabla \eta
+ \int_Q P(\bar \ph)(q-p) \eta = 0,
\end{alignat}
\end{subequations}
where we have used $\eta(0) = 0$ and $q(T) = 0$.
Moreover, the following identity has been employed
\begin{align*}
\int_\Omega (J \star \rho)\omega=\int_\Omega (J\star\omega)\rho \,,\qquad\forall \omega,\rho\in H\,.
\end{align*}
Then, \eqref{cor:opt} is a consequence of comparing the sum of \eqref{adj:lin:1}-\eqref{adj:lin:2} and the sum of \eqref{adj:lin:3}-\eqref{adj:lin:4}.
\end{proof}

\subsection{G\^ateaux differentiability of the solution mapping}
The main result of this section is the following G\^ateaux differentiability of the solution mapping.

\begin{thm}\label{thm:Gat}
Assume that \eqref{ass:m:F}--\eqref{ass:ini}, \eqref{ass:str:m:n}--\eqref{ass:str:P} are satisfied. Fix $\bar \ph_0, v \in U$, $\s_0 \in V$, and set $[\bar \ph, \bar \s] = \mathcal{S}(\bar \ph_0, \s_0)$ as the unique strong solution to $(\mathrm{P})$ corresponding to initial data $[\bar \ph_0, \s_0]$ obtained from Theorem \ref{thm:strong}.  Denote by $h = v - \bar \ph_0 \in L^\infty(\Omega) \cap (U - U)$, and $[\xi, \eta]$ as the unique weak solution to the linearised system \eqref{linearised} corresponding to $(h, \bar \ph, \bar \s)$.  Then, the solution mapping $\mathcal{S}$ is weakly directionally differentiable at $\bar \ph_0$ in the direction $h = v - \bar \ph_0$, and as $\tau \downarrow 0$,
\begin{align*}
\frac{\mathcal{S}(\bar \ph_0 + \tau h, \s_0) - \mathcal{S}( \bar \ph_0, \s_0)}{\tau} \rightharpoonup [\xi, \eta] \text{ in } \Big (  H^1(0,T;V^*) \cap L^\infty(0,T;H) \cap L^2(0,T;V) \Big )^2.
\end{align*}
\end{thm}

\begin{proof}
Let $\last{ \{\tau_k\}_{k \in \mathbb{N}} }\subset (0,1]$ be a null sequence, and by convexity of $U$, it holds that
\begin{align*}
u_{\tau_k} := (1- \tau_k) \bar \ph_0 + \tau_{k} v \in U \quad \forall k \in \mathbb{N}.
\end{align*}
For all $k \in \mathbb{N}$, by Theorem \ref{thm:strong} $(\mathrm{P})$ admits a unique solution $[\ph_{\tau_k}, \s_{\tau_k}]$ corresponding to the initial data $[u_{\tau_k}, \s_0]$.  Then, invoking Theorem \ref{thm:cts1} we find that
\begin{align*}
& \| \ph_{\tau_k} - \bar \ph \|_{ H^1(0,T;V^*) \cap L^\infty(0,T;H) \cap L^2(0,T;V) } + \| \s_{\tau_k} - \bar \s \|_{ H^1(0,T;V^*) \cap L^\infty(0,T;H) \cap L^2(0,T;V) }\\
& \quad \leq K_4 \tau_k \| v - \bar \ph_0 \|.
\end{align*}
This shows that
\begin{align*}
\ph_{\tau_k} \to \bar \ph, \quad \s_{\tau_k} \to \bar \s & \text{ in } H^1(0,T;V^*)\cap L^\infty(0,T;H) \cap L^2(0,T;V), \\
\tau_{k}^{-1}(\ph_{\tau_k} - \bar \ph) \rightharpoonup \widehat \Phi, \quad \tau_{k}^{-1}(\s_{\tau_k} - \bar \s) \rightharpoonup \widehat \Sigma & \text{ in }
H^1(0,T;V^*)\cap L^\infty(0,T;H) \cap L^2(0,T;V) , \\
\tau_{k}^{-1}(\ph_{\tau_k} - \bar \ph) \to \widehat \Phi, \quad \tau_{k}^{-1}(\s_{\tau_k} - \bar \s) \to \widehat \Sigma & \text{ in } C^0([0,T];V^*) \cap L^2(0,T;H),
\end{align*}
with $\widehat \Phi, \widehat \Sigma$ also belonging to $C^0([0,T];H)$ due to the continuous embedding $H^1(0,T;V^*)\cap L^2(0,T;V) \subset C^0([0,T];H)$ (i.e., the Lions--Magenes lemma).  Invoking the integral form of the mean value theorem, for any $v \in W$ and $\zeta \in C^\infty_c(0,T)$, we find that
\begin{align*}
& \frac{1}{\tau_k} \int_Q  ( P(\ph_{\tau_k}) - P(\bar \ph) )(\s_{\tau_k} + B J \star \ph_{\tau_k})\zeta(t) v \\
& \quad = \int_Q \int_0^1 P'((1-s) \ph_{\tau_k} + s \bar \ph) ds \frac{\ph_{\tau_k} - \bar \ph}{\tau_k} [(\s_{\tau_k} + B J \star \ph_{\tau_k}) \zeta(t) v] \\
& \quad \to \intQ P'(\bar \ph) (\bar \s + B J \star \bar \ph) \widehat \Phi \zeta(t) v,
\end{align*}
due to the strong convergence of $\tau_k^{-1}(\ph_{\tau_k} - \bar \ph)$ in $L^2(0,T;H)$, as well as the strong convergence of the remainder term in $L^2(0,T;H)$ by the generalised dominated convergence theorem.  Similarly, we see that
\begin{align*}
& \frac{1}{\tau_k} \int_Q \zeta(t) (\nabla \Lam(\ph_{\tau_k}) - \nabla \Lam(\bar \ph) ) \cdot \nabla v \\
& \quad = \frac{A}{\tau_k} \int_Q \zeta(t) [(\lam(\ph_{\tau_k}) - \lam( \bar \ph))\nabla \ph_{\tau_k} + \lam(\bar \ph) \nabla (\ph_{\tau_k} - \bar \ph)] \cdot \nabla v \\
& \quad = A \int_Q  \int_0^1 \lam'((1-s) \ph_{\tau_k} + s \bar \ph) ds \frac{\ph_{\tau_k} - \bar \ph}{\tau_k} [\nabla \ph_{\tau_k} \cdot \zeta(t) \nabla v] + \lam(\bar \ph) \nabla \Big (\frac{\ph_{\tau_k} - \bar \ph}{\tau_k} \Big )\cdot \zeta(t) \nabla v \\
& \quad \to A \int_Q \big(\lam'(\bar \ph) \widehat \Phi \nabla \bar \ph + \lam(\bar \ph) \nabla \widehat \Phi \big) \cdot \zeta(t) \nabla v,
\end{align*}
as well as
\begin{align*}
& \frac{1}{\tau_k} \int_Q \zeta(t) B \Big ((m(\ph_{\tau_k}) - m(\bar \ph)) (\nabla J \star \ph_{\tau_k}) + m(\bar \ph) (\nabla J \star (\ph_{\tau_k} - \bar \ph)) \Big )\cdot \nabla v \\
& \quad \to \int_Q B [m'(\bar \ph) \widehat \Phi (\nabla J \star \bar \ph) + m(\bar \ph) (\nabla J \star \widehat \Phi)] \cdot \zeta(t) \nabla v.
\end{align*}
In particular, by subtracting the weak formulation \eqref{W_first}-\eqref{W_second} of $(\mathrm{P})$ for $[\bar \ph, \bar \s]$ with test functions $v, w \in W$ from the weak formulation of $(\mathrm{P})$ for $[\ph_{\tau_k}, \s_{\tau_k}]$, dividing the resulting equations by $\tau_k$, multiplying by $\zeta(t) \in C^\infty_c(0,T)$, integrating over $(0,T)$ and then sending $\tau_k \to 0$, we infer that $[\widehat \Phi, \widehat \Sigma]$ satisfies \eqref{lin:weak:1}-\eqref{lin:weak:2} (with $\widehat{\Phi}$ and $\widehat{\Sigma}$ in place of $\xi$ and $\eta$, respectively), after applying a standard argument to remove the integral over $(0,T)$ and applying the density of $W$ in $V$.  Moreover, using that $\tau_{k}^{-1}(\ph_{\tau_k} - \bar \ph) \to \widehat \Phi$ in $C^0([0,T];V^*)$ we find that
\begin{align*}
\< v - \bar \ph_0, \theta>_V = \< \tau_{k}^{-1}( \ph_{\tau_k}(0) - \bar \ph_0), \theta >_V \to  \< \widehat \Phi(0), \theta >_V \quad \forall \theta \in V.
\end{align*}
Since $\widehat \Phi \in C^0([0,T];H)$ we obtain the identification $\widehat \Phi(0) = h = v - \bar \ph_0$ and a similar argument also shows $\widehat \Sigma(0) = 0$.  Hence, $[\widehat \Phi, \widehat \Sigma]$ is a solution to the linearised system \eqref{linearised} corresponding to $(h, \bar \ph, \bar \s)$, and thus by uniqueness we can identify $\widehat \Phi = \xi$ and $\widehat \Sigma = \eta$.  In light of the fact that $[\xi, \eta]$ is independent of the subsequence $\last{ \{\tau_k\}_{k \in \mathbb{N}} }$
chosen for the weak/strong convergences, we conclude the assertion is valid for any null sequence.
\end{proof}

\subsection{Analysis of the constrained minimisation problem}
The main result concerning \eqref{opt} is formulated as follows.

\begin{thm}
\label{thm:opt}
Suppose \eqref{ass:m:F}-\eqref{ass:ini}, \eqref{ass:str:m:n}-\eqref{ass:str:P} are fulfilled.
Let $U$ be defined by \eqref{defU} and let $\ph_\Omega \in H$.  Then, the following assertions hold:
\begin{enumerate}
\item[$(i)$]  For any $\alpha > 0$, there exists at least one solution $\bar \ph_0^\alpha \in U$ to the minimisation problem \eqref{opt}.
\item[$(ii)$] Suppose the inverse problem \eqref{inv} has at least one admissible solution $\ph_*$, i.e., $\exists \, \ph_* \in U $ such that $S(\ph_*) = \ph_\Omega$ a.e.~in $\Omega$.  
    \last{Let $\{\alpha_\delta\}_{\delta > 0}$ be a sequence of positive real numbers such that
    $\alpha_\delta\to 0$, and
     $\delta^2/\alpha_\delta$ is bounded as $\delta \to 0$.}  Let $\bar \ph_0^{\alpha_\delta}$ be a solution to the constrained minimisation problem
\begin{align*}
\mathrm{argmin}_{u \in U} \Big ( \frac{1}{2} \| S(u) - \ph_\Omega^\delta \|^2 + \frac{\alpha_\delta}{2} \| u \|_V^2 \Big ),
\end{align*}
where $\ph_\Omega^\delta \in H$ satisfies $\| \ph_\Omega - \ph_\Omega^\delta \| \leq \delta$.  Then, there exists a non-relabelled subsequence of \last{$\{\bar \ph_0^{\alpha_\delta} \}_{\delta > 0}$} and a solution $\bar \ph_0^* \in U$ to the inverse problem \eqref{inv} such that, as $\delta \to 0$,
\begin{align}\label{inv:conv}
\bar \ph_0^{\alpha_\delta} \rightharpoonup \bar \ph_0^* \text{ in } V, \quad \bar \ph_0^{\alpha_\delta} \to \bar \ph_0^* \text{ in } H \text{ and a.e.~in } \Omega.
\end{align}
\item[$(iii)$] For any $\alpha > 0$, assume $\ph_\Omega \in V$, and denote by $[\bar \ph, \bar \s]$ the unique strong solution to $(\mathrm{P})$ corresponding to initial data $[\bar \ph_0^\alpha, \s_0]$ obtained from Theorem \ref{thm:strong}, and by $[p, q]$ the unique weak solution to \eqref{adjoint} corresponding to $(\ph_\Omega, \bar \ph, \bar \s)$.  Then, it holds that
\begin{align}\label{opt:ineq}
\iO (\alpha \bar \ph_0^\alpha  + p(0))(v - \bar \ph_0^\alpha) + \iO \alpha \nabla \bar \ph_0^\alpha \cdot \nabla (v - \bar \ph_0^\alpha) \geq 0 \quad \forall v \in U.
\end{align}
\end{enumerate}
\end{thm}
\begin{proof}
The first assertion on the existence of a solution to \eqref{opt} can be proved using the direct method, and since this is somewhat standard in the literature we omit the details.

For the second assertion we adapt the ideas in \cite[Prop.~2.5]{BRV}, see also \cite[Thm.~10.3]{EHN}.  From the definition of $\bar \ph_0^{\alpha_\delta}$ we see that
\begin{align}\label{inv:comp}
\frac{1}{2} \| S( \bar \ph_0^{\alpha_\delta}) - \ph_\Omega^{\delta} \|^2 + \frac{\alpha_\delta}{2} \| \bar \ph_0^{\alpha_\delta} \|_V^2 \leq \frac{1}{2} \| S(\ph_*) - \ph_\Omega^{\delta} \|^2 + \frac{\alpha_\delta}{2} \| \ph_* \|_V^2.
\end{align}
In particular, since $S(\ph_*) = \ph_\Omega$ and $\| \ph_\Omega - \ph_\Omega^{\delta} \| \leq \delta$, it holds that
\begin{align*}
\| \bar \ph_0^{\alpha_\delta} \|_V^2 \leq \frac{2}{\alpha_\delta} \Big ( \frac{\delta^2}{2} + \frac{\alpha_\delta}{2} \| \ph_* \|_V^2 \Big ) = \frac{\delta^2}{\alpha_\delta} + \| \ph_* \|_V^2.
\end{align*}
This yields uniform boundedness of $\{ \bar \ph_0^{\alpha_\delta} \}_{\delta > 0}$ in $V$, and by compactness we infer the existence of a function $\bar \ph_0^* \in V$ such that \eqref{inv:conv} holds.  To see that $\bar \ph_0^*$ is a solution to the inverse problem, we first notice that $\bar \ph_0^*\in U$, since $U$ is closed in $H$, and we recall the continuity property $S(\bar \ph_0^{\alpha_\delta}) \to S(\bar \ph_0^*)$ in $H$, which comes from the continuous dependence estimate \eqref{cts1} in Theorem \ref{thm:cts1} and the Lions--Magenes lemma $ H^1(0,T;V^*) \cap L^2(0,T;V)  \subset C^0([0,T];H)$.  In particular, we deduce from \eqref{inv:comp} that
\begin{align*}
\| S(\bar \ph_0^{\alpha_\delta}) - \ph_\Omega^{\delta} \|^2 \leq \| S(\ph_*) - \ph_\Omega^{\delta} \|^2 + \alpha_\delta \| \ph_* \|_V^2 \leq \delta^2 + \alpha_\delta \| \ph_* \|_V^2 \to 0 \text{ as } \delta \to 0,
\end{align*}
and so passing to the limit $\delta \to 0$ results in
\begin{align*}
\| S(\bar \ph_0^*) - \ph_\Omega \|^2 \leq 0,
\end{align*}
which implies $S(\bar \ph_0^*) = \ph_\Omega$ a.e.~in $\Omega$.

For the last assertion, we look at the differentiability of the objective functional $f_\alpha$ given by
\begin{align*}
f_\alpha(u):=\frac{1}{2}\Vert S(u)-\varphi_\Omega\Vert^2+\frac{\alpha}{2}\Vert u\Vert_V^2\qquad\forall u\in U\,,
\end{align*}
at $\bar \ph_0$ in the direction $v - \bar \ph_0$.  Standard arguments show
\begin{align*}
\lim_{\tau \downarrow 0} \frac{1}{2\tau} \big ( \| \bar \ph_0 + \tau (v - \bar \ph_0) \|_V^2 - \| \bar \ph_0 \|_V^2 \big ) = \iO \bar \ph_0 (v - \bar \ph_0)
+ \iO \nabla \bar \ph_0 \cdot \nabla (v - \bar \ph_0).
\end{align*}
Moreover, denoting $[\ph_\tau, \s_\tau] = \mathcal{S}(\bar \ph_0 + \tau (v - \bar \ph_0), \s_0)$, and $[\xi, \eta]$ as the unique solution to \eqref{linearised} corresponding to $(h, \bar \ph, \bar \s)$, we also have
\begin{align*}
& \lim_{\tau \downarrow 0} \frac{1}{2\tau} \big ( \| \ph_{\tau}(T) - \ph_\Omega \|^2 - \| \bar \ph(T) - \ph_\Omega\|^2\big )\\
& \quad  = \lim_{\tau \downarrow 0} \frac{1}{2} \iO \Big(\tfrac{\ph_\tau(T) - \bar \ph(T)}{\tau} (\ph_\tau(T) - \ph_\Omega) + \tfrac{\ph_\tau(T) - \bar \ph(T)}{\tau} (\bar \ph(T) - \ph_\Omega) \Big)\\
& \quad =  \lim_{\tau \downarrow 0} \frac{1}{2} \iO \Big(\big ( \tfrac{\ph_\tau(T) - \bar \ph(T)}{\tau}  - \xi(T) \big ) (\ph_\tau(T) - \ph_\Omega) + \big (\tfrac{\ph_\tau(T) - \bar \ph(T)}{\tau} - \xi(T) \big ) (\bar \ph(T) - \ph_\Omega)\Big) \\
& \qquad + \iO \xi(T) (\bar \ph(T) - \ph_\Omega)
\end{align*}
in light of the strong convergence $\ph_\tau \to \bar \ph$ in $H^1(0,T;V^*)\cap L^2(0,T;V) \subset C^0(0,T;H)$, see Theorem \ref{thm:Gat}.  Due to the regularities stated in Theorem \ref{thm:strong} and invoking \cite[Sec.~5.9, Thm.~4]{EV} we have that
\begin{align*}
\ph_\tau, \bar \ph \in L^2(0,T;W) \cap H^1(0,T;H) \subset C^0([0,T];V),
\end{align*}
and from the proof of Theorem \ref{thm:strong}, there exists a positive constant $C$ not depending on $\ph_\tau$, $\s_\tau$, $\bar \ph_0 + \tau(v - \bar \ph_0)$ and $\s_0$ such that
\begin{align*}
\| \ph_\tau \|_{H^1(0,T:H) \cap L^\infty(0,T;V) \cap L^2(0,T;W)} \leq C \|\bar \ph_0 + \tau(v - \bar \ph_0) \|_V + C \| \s_0 \|_V.
\end{align*}
In particular we have the uniform boundedness of $\|\ph_\tau(T) \|_V$ in $\tau \in (0,1]$.  Hence, as $\tau \to 0$,
\begin{align*}
& \left | \iO \big ( \tfrac{\ph_\tau(T) - \bar \ph(T)}{\tau}  - \xi(T) \big ) (\ph_\tau(T) - \ph_\Omega) \right | \\
& \quad  \leq  \| \tfrac{\ph_\tau(T) - \bar \ph(T)}{\tau} - \xi(T) \|_{V^*} \| \ph_\tau(T) - \ph_\Omega \|_V \leq C \| \tfrac{\ph_\tau - \bar \ph}{\tau} - \xi \|_{C^0([0,T];V^*)} \to 0,
\end{align*}
which implies
\begin{align*}
 \lim_{\tau \downarrow 0} \frac{1}{2\tau} \big ( \| \ph_{\tau}(T) - \ph_\Omega \|^2 - \| \bar \ph(T) - \ph_\Omega\|^2\big ) = \iO \xi(T) (\bar \ph(T) - \ph_\Omega ).
\end{align*}
Hence, by applying \cite[Lem.~2.21]{Tr} to the objective functional $f_\alpha:U\to\mathbb{R}$, we deduce that
$\bar \ph_0 \in U$ necessarily satisfies the optimality condition
\begin{align*}
\alpha \iO \bar \ph_0 (v - \bar \ph_0) + \alpha \iO \nabla \bar \ph_0 \cdot \nabla ( v - \bar \ph_0) + \iO \xi(T) (\bar \ph(T) - \ph_\Omega) \geq 0 \quad \forall v \in U,
\end{align*}
and through \eqref{cor:opt} we obtain \eqref{opt:ineq}.
\end{proof}

\section*{Acknowledgments}
\last{The authors are grateful to the reviewers for their comments and suggestions.}
The first and third authors are members of GNAMPA (Gruppo Nazionale per l'Analisi Matematica, la Probabilit\`{a} e le loro
Applicazioni) of INdAM (Istituto Nazionale di Alta Matematica).
The work of the second author is partially supported by a grant from the Research Grants Council of the Hong Kong Special Administrative Region, China [Project No.: HKBU 14302319].

\end{document}